\documentclass[final,3p,times]{elsarticle}

\usepackage{amssymb}
\usepackage{bm}
\usepackage{amsmath}
\usepackage{subcaption}
\usepackage{amsthm}
\usepackage{ulem}

\newtheorem{theorem}{Theorem}
\newtheorem{lemma}{Lemma}
\newtheorem{corollary}{Corollary}
\newtheorem{definition}{Definition}

\journal{A journal}

\begin{document}

\begin{frontmatter}



\title{On the Role of Split Formulations on Aliasing Errors and Entropy Stability of Discontinuous Galerkin Schemes}

\author[label1]{Mathias Dufresne-Piché}
\author[label1]{Siva Nadarajah}
\affiliation[label1]{organization={McGill University},
            addressline={845 Sherbrooke St W},
            city={Montreal},
            postcode={H3A 0G4},
            state={Quebec},
            country={Canada}}

\begin{abstract}
In this work, we formally investigate the dealiasing properties of split form discontinuous Galerkin (DG) discretizations for the one-dimensional Burgers problem. By generalizing the proof of Blaisdell et al. {\cite{blaisdell1996effect}} for spectral discretizations, we show that split form DG schemes achieve dealiasing in the weakly underresolved range through integration error cancellation on the conservative and non-conservative forms. As a corollary, we identify quadrature- and order-dependent pairs of splitting coefficients that eliminate the dominant component of the aliasing error. While these ``optimized'' splitting coefficients minimize integration errors on the numerical scheme, we show that the skew-symmetric aliasing pattern resulting from the entropy stable split is required to maintain long-term stability of the numerical solution. {The proposed framework provides an alternative proof for the entropy stability of the DG discretization of the (1/3, 2/3) split formulation introduced by Gassner \cite{gassner2013skew} which clarifies the connection between DG aliasing errors and entropy stability.} Finally, we also show that the entropy stable split is associated with lower aliasing errors in the weakly underresolved range compared to the conservative form.
\end{abstract}




\end{frontmatter}



\section{Introduction}
The practical relevance of split form discretizations has long been recognized in the computational fluid dynamics (CFD) community. In general, discrete split forms are obtained as weighted averages of discretizations of the conservative and non-conservative formulations of a hyperbolic conservation law. While at the continuous level these formulations are mathematically identical, discretizations of the latter yield distinct numerical properties \cite{10.1007/978-3-642-46395-2_29}. In \cite{ZANG199127}, Zang observed numerically, that split form spectral discretizations for the non-linear terms of the incompressible Navier-Stokes equations resulted in reduced aliasing errors. This empirical observation was formalized in the work of \cite{blaisdell1996effect}. Through careful analysis of the one-dimensional Burgers problem, they proved that split form spectral discretizations of quadratic non-linearities achieved dealiasing in the weakly underresolved range by leveraging cancellation of aliasing errors resulting from conservative and non-conservative discretizations. These findings were further investigated by \cite{CHOW2003366} for finite difference schemes in the context of incompressible large-eddy simulations.

In the finite difference and finite volume frameworks, split formulations are also well-known for enabling the construction of skew-symmetric discretizations of conservation laws that allow for the enforcement of provable discrete entropy stability guarantees \cite{jameson2008construction, GERRITSEN1996245}. Entropy stability is particularly relevant for problems involving shock formation as it offers a natural approach to ensure physicality of weak numerical solutions. More recently, these ideas have been extended to the discontinuous Galerkin (DG) framework and used to create entropy stable split form DG schemes for the Burgers equation \cite{gassner2013skew} and the advection equation with variable wave-speed \cite{doi:10.1137/130928650}. Following the approach of \cite{FISHER20113727}, it is possible to rewrite split DG formulations in terms of high-order summation-by-parts (SBP) operators acting on sub-cell two-point flux functions \cite{GASSNER201639}. Using the flux differencing (FD) framework of \cite{chan2018discretely, chan2019skew}, this allows for the construction of split form DG and entropy stable DG discretizations for general hyperbolic conservation laws under arbitrary choices of volume and surface quadrature rules. The FD formulation was later extended to encompass energy stable flux reconstruction schemes (ESFR) \cite{CASTONGUAY2013400} through the work of \cite{CICCHINO2022111094, CICCHINO2025113532}.

It is well known that aliasing errors incurred by DG discretizations are directly tied to scheme stability and numerical solution quality. While the impact of quadrature strength on DG aliasing errors has been thoroughly studied in \cite{MENGALDO201556}, to the knowledge of the authors of this work, a complete theoretical description of the aliasing errors associated with split form DG schemes is still lacking. Empirical evidence \cite{WINTERS20181} suggests that split DG formulations for the Burgers and the Euler equations achieve aliasing error reduction akin to split form spectral discretizations. A formal proof for this phenomenon is, however, yet to be proposed. It should be noted that the impact of aliasing errors on numerical solution stability has been investigated by \cite{manzanero2018insights} for split form DG schemes of the advection equation with variable wave-speed. This work is, however, restricted to linear problems and merely focuses on the temporal growth of aliasing-driven instabilities and hence does not characterize the modal content of the associated aliasing errors.

In this paper, we aim to close this theoretical gap by proposing a generalization of the work \cite{blaisdell1996effect} on split form spectral discretizations to the DG framework. More precisely, we provide a full characterization of aliasing errors for split form DG schemes for the one-dimensional Burgers problem for arbitrary quadrature rules and therefore provide a formal proof for the behaviour observed by \cite{WINTERS20181}. By studying the modal structure of aliasing errors incurred by split form DG schemes, we also clarify the connection between discrete entropy stability and dealiasing for the one-dimensional Burgers problem.

This paper is organized as follows. In Section \ref{subsec:prob_statement}, we first show how aliasing errors for split DG discretizations of quadratic non-linearities can be described by a single rank 3 tensor. Sections \ref{subsec:proj_struct} and \ref{subsec:aliasing_error} describe how the structure of this tensor is affected by different choices of quadrature rules and splitting coefficients. In Section \ref{subsec:E_stab}, the specific tensor structure associated with entropy stable split forms is discussed. Sections \ref{sec:T_integration} and \ref{sec:inexact_mass} briefly describe the effects of temporal integration of aliasing errors and the impacts of using an inexact mass matrix in the DG discretization. Finally, in Section \ref{sec:num_exp}, our theoretical results are validated via numerical experiments.

\section{Aliasing Error Analysis for Quadratic Split Form DG Discretizations}

\subsection{Problem Statement for the 1D Burgers Equation}
\label{subsec:prob_statement}
We are concerned with split form DG spatial discretizations for the one-dimensional Burgers equation
\begin{equation}
    \frac{\partial u}{\partial t}+\frac{1}{2}\frac{\partial u^2}{\partial x} = 0,
    \quad
    x \in \Omega,
    \quad
    t \in [0, T],
\end{equation}
where $T > 0$ and $\Omega := [-1,1]$ denotes the reference element. These discretizations of interest are obtained from a weighted average of the conservative and non-conservative DG formulations. In matrix form, using the FD framework \cite{chan2019skew, chan2018discretely}, they can be written as
\begin{equation}
    \mathbf{M}\frac{d\mathbf{u}}{dt} +
    \begin{bmatrix}
        \bm{\chi}_q \\
        \bm{\chi}_f
    \end{bmatrix}^T
    ((\mathbf{Q}_h - \mathbf{Q}_h^T) \circ \mathbf{F}) \mathbf{1}
    + \bm{\chi}_f^T \mathbf{B} \mathbf{f}^* = 0,
    \label{eq:FD_formulation}
\end{equation}
where $\mathbf{f}^*$ is the numerical flux and $\mathbf{F}$ is the two-point flux matrix defined as
\begin{equation}
    \mathbf{F}_{ij} := F(u_i, u_j) := \frac{1}{4}\alpha(u_i^2 + u_j^2) + \frac{1}{2}\beta u_i u_j
    \label{eq:2ptflux}
\end{equation}
with splitting coefficients $\alpha$, $\beta$ chosen such that $\alpha + \beta = 1$. In Eq.(\ref{eq:FD_formulation}), $\mathbf{M}$ denotes the mass matrix, $\mathbf{Q}_h - \mathbf{Q}_h^T$ is a general skew-symmetric operator, $\mathbf{B}$ is the boundary integration matrix and $\bm{\chi}_q$ and $\bm{\chi}_f$ represent interpolation matrices at the volume and face quadrature nodes respectively. We refer the reader to the work of \cite{chan2019skew, chan2018discretely} for a complete description of the flux differencing formulation.

For two-point fluxes given by Eq.(\ref{eq:2ptflux}), it is straightforward to show (see \ref{app1}) that Eq.(\ref{eq:FD_formulation}) is mathematically equivalent to the semi-discrete weak formulation
\begin{align}
    \int_\Omega \frac{d}{dt}\Pi(\phi u) d\Omega
    - \alpha \frac{1}{2}\int_\Omega \frac{d \phi}{dx} \Pi(u^2) d\Omega
    + \beta \int_{\Omega} \frac{du}{dx}\Pi(\phi u) d\Omega
    - \beta \frac{1}{2}\int_{\Gamma} \phi u^2 \hat{n} d\Gamma
    + \int_{\Gamma} \phi f^* \hat{n} d\Gamma
    = 0,
    \label{eq:FD_continuous}
\end{align}
where $\Gamma := \partial\Omega$ and $\hat{n}$ is the outward normal on $\Gamma$. In Eq.(\ref{eq:FD_continuous}), $\phi, u \in \mathcal{P}^p$ and $\Pi : L2 \to \mathcal{P}^q$ is the interpolation operator at the $(q+1)$ volume quadrature nodes. It is assumed that volume quadrature nodes are chosen such that the strength of the underlying quadrature is at least $2p-1$. Volumetric aliasing effects on the quadratic non-linearity can hence be investigated by studying the term
\begin{equation}
    - \alpha \frac{1}{2}\int_\Omega \frac{d \phi}{dx} \Pi(u^2) d\Omega
    + \beta \int_{\Omega} \frac{du}{dx}\Pi(\phi u) d\Omega.
    \label{eq:vol_terms}
\end{equation}
It should be noted that when the mass matrix is inexact (as is the case when using Gauss-Legendre-Lobatto volume quadrature nodes), aliasing errors will also be introduced by the integral of the time derivative term. This will be discussed in more detail in section \ref{sec:inexact_mass}.

Without loss of generality, we express $u$ as
\begin{equation}
    u = \sum_{i=0}^p u_i L_i,
\end{equation}
where $L_i$ denotes the $i$th Legendre polynomial, and consider writing Eq.(\ref{eq:vol_terms}) in the Legendre basis. For $0 \leq k \leq p$, this yields
\begin{align}
    - &\alpha \frac{1}{2}\int_\Omega \frac{d L_k}{dx} \Pi(u^2) d\Omega
    + \beta \int_{\Omega} \frac{du}{dx}\Pi(L_k u) d\Omega \nonumber \\
    &=
    -\sum_{i,j=0}^p \frac{1}{2}u_iu_j\alpha \int_\Omega \frac{d L_k}{dx} \Pi(L_i L_j) d\Omega
    +
    \sum_{i,j=0}^p u_iu_j \beta \int_\Omega \frac{d L_i}{dx} \Pi(L_k L_j) d\Omega \nonumber \\
    &=\sum_{i,j=0}^p \frac{1}{2}u_iu_j \int_\Omega \left(\beta \frac{d L_i}{dx} \Pi(L_j L_k)
    +
    \beta \frac{d L_j}{dx} \Pi(L_i L_k)-\alpha \frac{d L_k}{dx} \Pi(L_i L_j)
    \right)d\Omega,
    \label{eq:vol_terms_exp}
\end{align}
where symmetry of the product expansion has been used in the final step. Consequently, assuming that no fortuitous cancellation occurs among the summands of Eq.(\ref{eq:vol_terms_exp}), aliasing errors on the quadratic non-linearity are fully characterized by the error tensor
\begin{align}
    \Delta_{ijk}(\alpha, \beta) :=
    \int_\Omega \frac{dL_k}{dx}L_iL_j d\Omega
    -\beta \int_\Gamma L_iL_jL_k \hat{n}d\Gamma
    +
    \int_\Omega \left(\beta \frac{d L_i}{dx} \Pi(L_j L_k)
    +
    \beta \frac{d L_j}{dx} \Pi(L_i L_k)-\alpha \frac{d L_k}{dx} \Pi(L_i L_j)
    \right)d\Omega.
    \label{eq:Delta_def}
\end{align}
Provided that $\alpha + \beta=1$, this can be equivalently written as
\begin{align}
    \Delta_{ijk}(\alpha, \beta)
    &=
   {
    \alpha\int_\Omega \frac{dL_k}{dx}L_iL_j d\Omega
    + \beta\int_\Omega \frac{dL_k}{dx}L_iL_j d\Omega
    -\beta \int_\Gamma L_iL_jL_k \hat{n}d\Gamma \nonumber} \\
    &{\quad+
    \int_\Omega \left(\beta \frac{d L_i}{dx} \Pi(L_j L_k)
    +
    \beta \frac{d L_j}{dx} \Pi(L_i L_k)-\alpha \frac{d L_k}{dx} \Pi(L_i L_j)
    \right)d\Omega \nonumber} \\
    &=
    \int_\Omega \left(\alpha \frac{d L_k}{dx} L_i L_j - \beta \frac{d L_i}{dx} L_j L_k
    -
    \beta \frac{d L_j}{dx} L_i L_k
    \right)d\Omega \nonumber\\
    &\quad+
    \int_\Omega \left(\beta \frac{d L_i}{dx} \Pi(L_j L_k)
    +
    \beta \frac{d L_j}{dx} \Pi(L_i L_k)-\alpha \frac{d L_k}{dx} \Pi(L_i L_j)
    \right)d\Omega,
\end{align}
{where integration by parts has been used. Defining}
\begin{equation}
    \tilde{\delta}_{lmn} := \int_\Omega \frac{d L_l}{dx} \Pi(L_m L_n)d\Omega - \int_\Omega \frac{d L_l}{dx} L_m L_nd\Omega,
    \label{eq:tdelta_def}
\end{equation}
one finally obtains
\begin{equation}
    \Delta_{ijk}(\alpha, \beta) = \beta \tilde{\delta}_{ijk} + \beta \tilde{\delta}_{jki} - \alpha \tilde{\delta}_{kij}.
    \label{eq:Delta_components}
\end{equation}
{From Eq.(\ref{eq:Delta_components}), it is clear that $\Delta_{ijk}(\alpha, \beta)$ measures the integration error on the non-linear flux resulting from the use of an inexact quadrature for the evaluation of the DG volumetric terms. However, by virtue of the preceding discussion, it can be seen that this integration error effectively stems from sampling of the non-linear flux at the volume quadrature nodes. Hence, integration error can equivalently be regarded as aliasing error in the same sense as that intended by \cite{blaisdell1996effect} in their study of spectral methods.}

It should finally be noted that while we used the one-dimensional Burgers equation to derive Eq.(\ref{eq:Delta_def}), the same error term should arise in the volume component of any quadratic non-linearity. We will spend the rest of this section characterizing the properties of $\Delta_{ijk}(\alpha, \beta)$ for general quadrature rules and use this to show that split form DG schemes are associated with an inherently lower aliasing error in the weakly underresolved range.

\subsection{Projection Structure}
\label{subsec:proj_struct}
To keep the discussion as general as possible, {we refrain from choosing a specific set of interpolation nodes for $\Pi$ and consider instead equipping the latter with the minimal structure required to conduct our analysis. This is formalized through Definition \ref{def:projection}.
\begin{definition}
    We say that a collocation projection $\Pi : L2 \to \mathcal{P}^q$ is valid if the following hold:
    \begin{enumerate}
    \item Quadrature nodes are symmetrically distributed in $\Omega$.
    \item The interpolation nodes associated with $\Pi$ induce a quadrature rule of strength $\tau \geq 2p-1$ on $\Omega$, \textit{ie.}, for $f \in \mathcal{P}^\tau$,
    \begin{equation}
        \int_{\Omega} fd\Omega = \int_\Omega \Pi(f)d\Omega.
    \end{equation}
\end{enumerate}
\label{def:projection}
\end{definition}
Note that by construction, the interpolations associated with the Gauss-Legendre (GL) and Gauss-Legendre-Lobatto (GLL) quadrature rules are valid collocation projections}. Definition \ref{def:projection} significantly constrains the structure of $\Pi$ when expressed in the Legendre basis, as shown in Lemma \ref{lem:proj_struct}. {
\begin{lemma}
    Let $\Pi$ be a valid collocation projection associated with a quadrature rule of strength $\tau$. Then,
    \begin{equation}
        \frac{1}{||L_i||^2}\int_{\Omega} L_i \Pi(L_j) d \Omega
        =
        \delta_{ij},
        \quad
        \text{for } i+j < \tau + 1,
        \label{eq:proj_resolved}
    \end{equation}
    where $\delta_{ij}$ is the Kronecker delta. Additionally, if $i + j$ is odd,
    \begin{equation}
        \frac{1}{||L_i||^2}\int_{\Omega} L_i \Pi(L_j) d \Omega
        =
        0.
        \label{eq:proj_parity}
    \end{equation}
    \label{lem:proj_struct}
\end{lemma}
\begin{proof}
    For Eq.(\ref{eq:proj_resolved}), we note that since $i+j < \tau + 1$,
    \begin{equation}
        \frac{1}{||L_i||^2}\int_{\Omega} L_i \Pi(L_j) d \Omega
        =
        \frac{1}{||L_i||^2}\int_{\Omega} L_i L_j d \Omega = \delta_{ij}.
    \end{equation}
    For Eq.(\ref{eq:proj_parity}), we observe that symmetry of interpolation nodes implies that $\Pi$ is parity-preserving. Hence, from the parity of Legendre polynomials, we see that $L_i\Pi(L_j)$ is odd if and only if $i+j$ is odd. The integral of the latter therefore vanishes on the symmetric domain $\Omega$.
\end{proof}
The behaviour described by Lemma \ref{lem:proj_struct} is exemplified in Figure \ref{fig:proj_fig}, where we plot the Legendre components of the image of $\Pi$ for GL and GLL quadrature nodes for the case $p=q=7$. As expected, for $i+j < \tau+1$ the $i$th Legendre component of $\Pi(L_j)$ coincides with the $L2$ projection. For $p < j \leq 2p$, aliasing occurs in a staggered fashion for components described by $i+j \geq \tau+1$. As this staggered pattern is not clearly visible in the GL case because the associated coefficients are small in magnitude, the aliased regime is plotted with tighter colour bounds in Figure \ref{fig:proj_fig_GL_zoom}. In what follows, we will refer to the Legendre components on the diagonal characterized by $i+j = \tau+1$ as the primary aliased band of $\Pi$ while components above this diagonal will be called the secondary aliased modes of $\Pi$.}
\begin{figure}
    \centering
    \begin{subfigure}{0.45\textwidth}
        \centering
        \includegraphics[trim=0cm 10cm 0cm 0cm, clip, width=\linewidth]{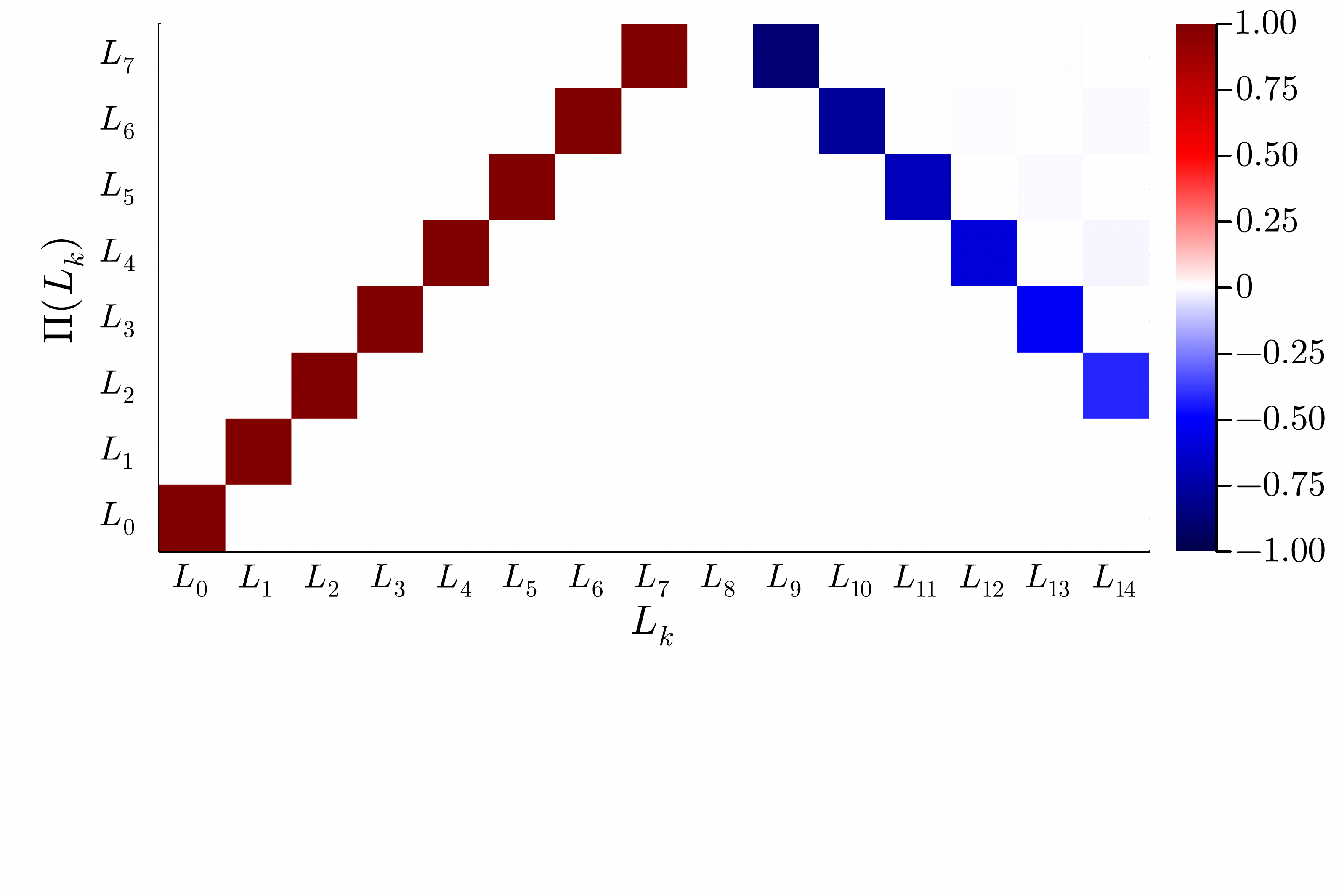}
        \caption{GL Quadrature.}
    \end{subfigure}
    \hfill
    \begin{subfigure}{0.45\textwidth}
        \centering
        \includegraphics[trim=0cm 10cm 0cm 0cm, clip, width=\linewidth]{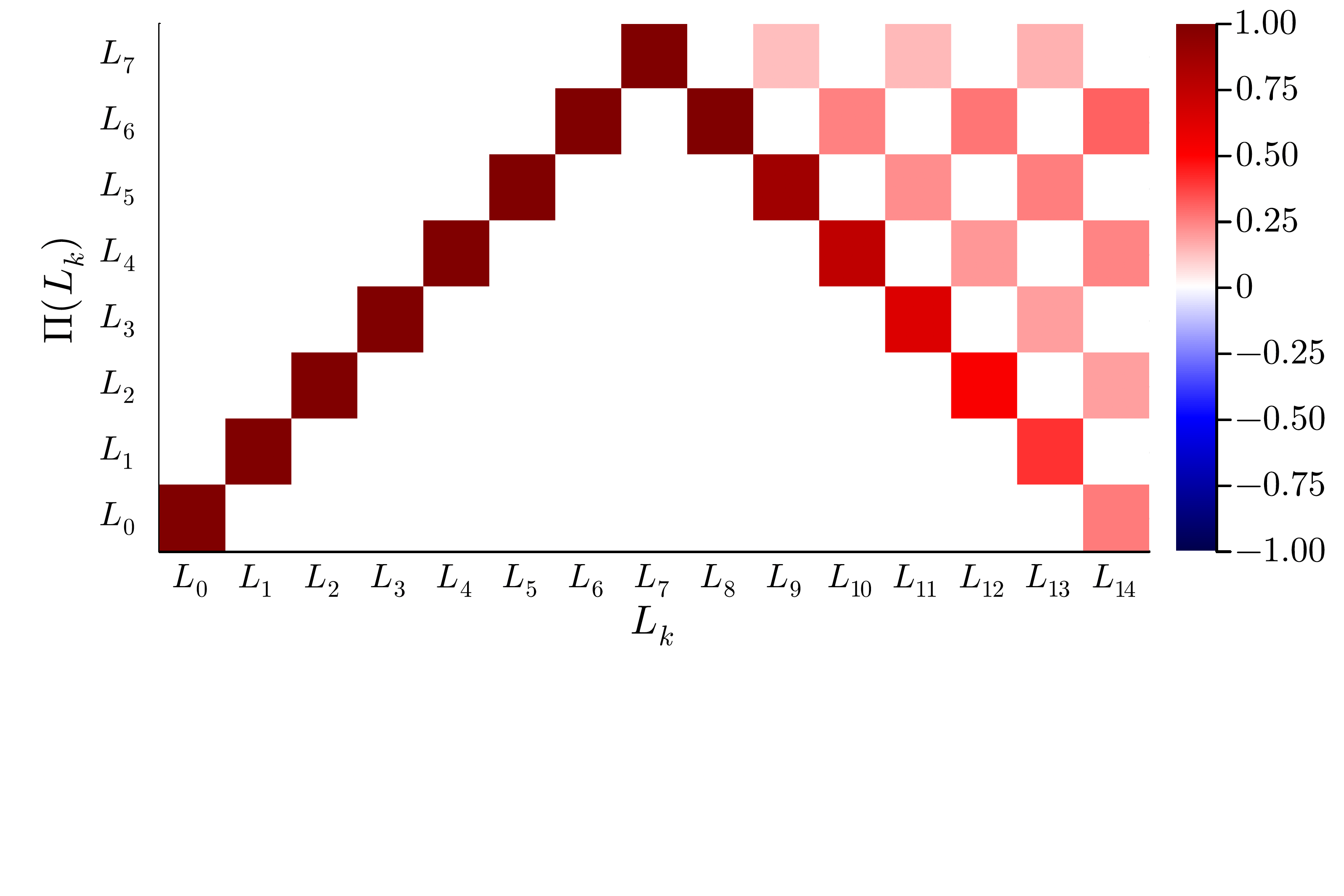}
        \caption{GLL Quadrature.}
    \end{subfigure}
    \caption{Structure of $\Pi$ in the Legendre basis using $(p+1)$ GL and GLL interpolation nodes ($p=q=7$).}
    \label{fig:proj_fig}
\end{figure}
\begin{figure}
    \centering
    \includegraphics[trim=0cm 10cm 0cm 0cm, clip, width=0.5\linewidth]{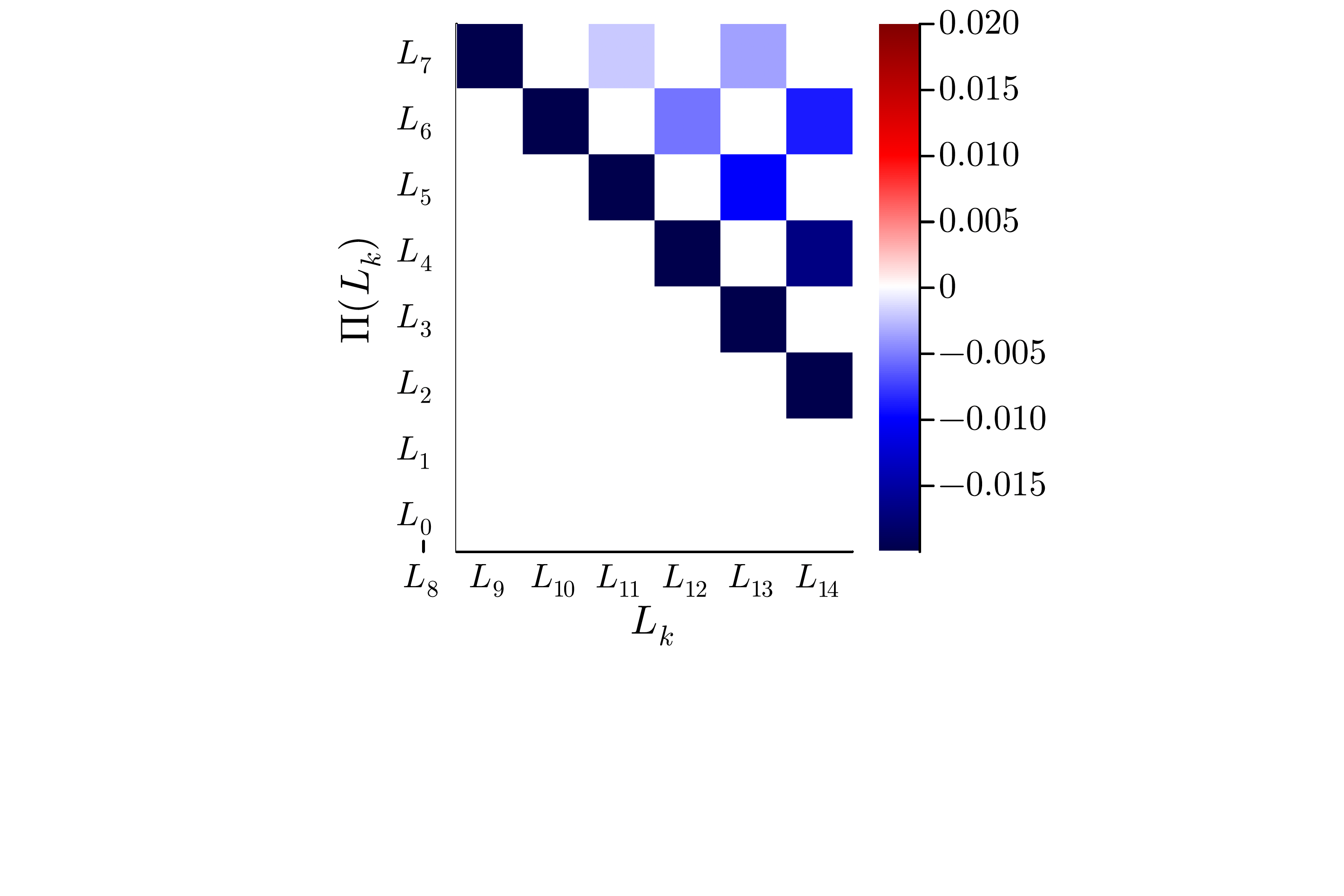}
    \caption{Structure of $\Pi$ in the Legendre basis using $(p+1)$ GL interpolation nodes ($p=q=7$). Colour scheme has been chosen to show the staggered structure of the secondary aliased modes.}
    \label{fig:proj_fig_GL_zoom}
\end{figure}
{As will be seen in the following, full characterization of the primary aliased band of $\Pi$ is necessary to describe the entries of $\Delta_{ijk}(\alpha, \beta)$ governing aliasing errors in the weakly underresolved range. To achieve this, it is convenient to make use of the following definition.
\begin{definition}
    Let $\Pi$ be a valid collocation projection operator with an associated quadrature strength of $\tau$. We define the primary aliasing coefficient of $\Pi$, denoted hereafter as $d_\Pi$ as
    \begin{equation}
    d_{\Pi} := \int_\Omega \Pi(L_{\tau+1}) d\Omega.
    \label{eq:d_Pi}
\end{equation}
\label{def:aliasing_coeff}
\end{definition}
The relevance of Definition \ref{def:aliasing_coeff} is justified by Lemma \ref{lem:proj_diag}, which shows that $d_\Pi$ is in fact sufficient to characterize the coefficient values of Legendre components associated with the primary aliased band of $\Pi$.}
\begin{lemma}
    {Assume that $\Pi$ is a valid collocation projection with an associated quadrature strength of $\tau$ and a primary aliasing coefficient $d_{\Pi}$.} Then, for $i+j=\tau+1$ and $i < p$,
    \begin{equation}
        \frac{1}{||L_i||^2}\int_{\Omega} L_i \Pi(L_j) d\Omega =
        \frac{(2\tau + 3)(2i+1)d_{\Pi}}{2}
        \begin{pmatrix}
            i & j & \tau+1 \\
            0 & 0 & 0
        \end{pmatrix}^2.
    \end{equation}
    \label{lem:proj_diag}
\end{lemma}
\begin{proof}
Since $\tau \geq 2p-1$ and $i < p$,
\begin{equation}
    \frac{1}{||L_i||^2}\int_{\Omega} L_i \Pi(L_j) d\Omega = \frac{1}{||L_i||^2}\int_{\Omega} \Pi(L_iL_j) d\Omega.
        \label{eq:lem_proj1}
\end{equation}
Using the product identity for Legendre polynomials, one obtains
\begin{equation}
    L_i L_j = (2\tau + 3)
    \begin{pmatrix}
        i & j & \tau+1 \\
        0 & 0 & 0
    \end{pmatrix}^2L_{\tau+1}
    +
    \sum_{r=|i-j|}^\tau
    (2r+1)
    \begin{pmatrix}
        i & j & r \\
        0 & 0 & 0
    \end{pmatrix}^2L_{r}.
\end{equation}
Using properties of $\Pi$, integrating over $\Omega$ and leveraging orthogonality of Legendre polynomials yields
\begin{equation}
    \int_{\Omega} \Pi(L_i L_j)d\Omega =
    (2\tau + 3)
    \begin{pmatrix}
        i & j & \tau+1 \\
        0 & 0 & 0
    \end{pmatrix}^2
    d_{\Pi}
    \label{eq:lem_proj2}.
\end{equation}
Combining Eq.(\ref{eq:lem_proj1}) and Eq.(\ref{eq:lem_proj2}) completes the proof.
\end{proof}
It should be noted that although this is not strictly required for the argument that follows, values of $d_{\Pi}$ can be computed analytically for a projection at the $(p+1)$ GL and GLL nodes by leveraging the kernel of the associated operator. In particular, in these special cases, as shown in \ref{app2}, one finds
\begin{align}
    d_{\Pi}^{GL} &=
    \int_{\Omega} \Pi(L_{2p+2})d\Omega
    =
    -\frac{2}{2p+3} \frac{
    \begin{pmatrix}
        4p+4 \\ 2p+2
    \end{pmatrix}
    }{
    \begin{pmatrix}
        2p+2 \\ p+1
    \end{pmatrix}^2
    }, \\
    d_{\Pi}^{GLL} &=
    \int_{\Omega} \Pi(L_{2p})d\Omega
    =
    \frac{2}{2p-1}
    \frac{
    \begin{pmatrix}
        4p \\ 2p
    \end{pmatrix}
    }
    {
    \begin{pmatrix}
        2p+2 \\ p+1
    \end{pmatrix}
    \begin{pmatrix}
        2p-2 \\ p-1
    \end{pmatrix}
    }.
\end{align}

\subsection{Aliasing Error}
\label{subsec:aliasing_error}
Having characterized the properties of the relevant collocation projection operators, we can now tackle the derivation of the aliasing error tensor $\Delta_{ijk}(\alpha, \beta)$. As a result of the properties of $\Pi$ when expressed in the Legendre basis, $\Delta_{ijk}$ inherits a parity-consistent hierarchical structure. More precisely, $\Delta_{ijk}(\alpha, \beta)$ vanishes when $i+j+k$ is even, and error modes are progressively excited as $i+j$ is increased from $\tau-p+2$ to $2p$.

In the following, we derive an expression for the error on the lowest-order Legendre mode, \textit{ie.}, for every pair $(i,j)$, we find the value of $\Delta_{ijk}(\alpha, \beta)$ for the smallest $k$ that results in a non-vanishing aliasing error. For the cases $i+j=\tau-p+2$ and $i+j=\tau-p+3$, since the aliasing error only manifests itself through a single Legendre mode, this amounts to a complete characterization. Following the same argument as \cite{blaisdell1996effect}, one can observe that in the weakly underresolved range, the error on the lowest-order product $L_i L_j$ will dominate. Hence, aliasing error reduction of a given quadratic split can be assessed by studying $\Delta_{ijk}(\alpha, \beta)$ when $i+j=\tau-p+2$.

We begin by deriving an expression for $\tilde{\delta}_{lmn}$ as shown in Lemma \ref{lem:tdelta}.
\begin{lemma}
    {Assume that $\Pi$ is a valid collocation projection with an associated quadrature strength of $\tau$ and a primary aliasing coefficient $d_{\Pi}$.} Then, if $0 \leq l,m,n \leq p$ and $l+m+n=\tau+2$,
    \begin{equation}
        \tilde{\delta}_{lmn} =
        \frac{d_{\Pi}l}{2}
        \frac{
        \begin{pmatrix}
            2l \\ l
        \end{pmatrix}
        \begin{pmatrix}
            2m \\ m
        \end{pmatrix}
        \begin{pmatrix}
            2n \\ n
        \end{pmatrix}
        }
        {
        \begin{pmatrix}
            2\tau +2 \\ \tau+1
        \end{pmatrix}
        },
        \label{eq:tdelta_binomial}
    \end{equation}
    where $\tilde{\delta}_{lmn}$ is defined as per Eq.(\ref{eq:tdelta_def}).
    \label{lem:tdelta}
\end{lemma}
\begin{proof}
    Using the product identity for Legendre polynomials,
    \begin{equation}
    L_m L_n = (2m+2n + 1)
    \begin{pmatrix}
        n & m & m+n \\
        0 & 0 & 0
    \end{pmatrix}^2L_{m+n}
    +
    \sum_{r=|m-n|}^{m+n-1}
    (2r+1)
    \begin{pmatrix}
        m & n & r \\
        0 & 0 & 0
    \end{pmatrix}^2L_{r}.
\end{equation}
Additionally, the derivative of $L_l$ is given by
\begin{equation}
    \frac{dL_l}{dx} = \frac{2}{||L_{l-1}||^2}L_{l-1} + \frac{2}{||L_{l-3}||^2}L_{l-3} + ...
\end{equation}
Since $l + m + n = \tau+2$, aliasing errors will only occur for the integration of the $(m+n)$th Legendre mode of the product $L_mL_n$ and the $(l-1)$th mode of the derivative of $L_{l-1}$. Thus,
\begin{align}
    \tilde{\delta}_{lmn} = \int_\Omega \frac{d L_l}{dx} \Pi(L_m L_n)d\Omega - \int_\Omega \frac{d L_l}{dx} L_m L_nd\Omega
    =
    (2m+2n+1)
    \begin{pmatrix}
        n & m & m+n \\
        0 & 0 & 0
    \end{pmatrix}^2
    \frac{2}{||L_{l-1}||^2}
    \int_{\Omega} L_{l-1} \Pi(L_{m+n})d\Omega.
\end{align}
Using Lemma \ref{lem:proj_diag}, we find
\begin{equation}
\tilde{\delta}_{lmn}
=
    {(2\tau + 3)(2l-1)(2m+2n+1)d_{\Pi}}
        \begin{pmatrix}
            l-1 & m+n & \tau+1 \\
            0 & 0 & 0
        \end{pmatrix}^2
        \begin{pmatrix}
            m & n & m+n \\
            0 & 0 & 0
        \end{pmatrix}^2.
\end{equation}
Upon simplification of the Wigner 3$j$ symbols via Lemma \ref{lem:Wigner1} (\ref{app3}), {one obtains
\begin{equation}
    \tilde{\delta}_{lmn} = (2l-1)d_\Pi
    \frac{
    \begin{pmatrix}
        2l-2\\ l-1
    \end{pmatrix}
    \begin{pmatrix}
        2m \\ m
    \end{pmatrix}
    \begin{pmatrix}
        2n \\ n
    \end{pmatrix}
    }{
    \begin{pmatrix}
        2\tau+2 \\ \tau+1
    \end{pmatrix}
    }
    =
    \frac{d_{\Pi}l}{2}
        \frac{
        \begin{pmatrix}
            2l \\ l
        \end{pmatrix}
        \begin{pmatrix}
            2m \\ m
        \end{pmatrix}
        \begin{pmatrix}
            2n \\ n
        \end{pmatrix}
        }
        {
        \begin{pmatrix}
            2\tau +2 \\ \tau+1
        \end{pmatrix}
        },
\end{equation}
where the shift identity for the central binomial coefficient has been used in the final step.}
\end{proof}
Leveraging Lemma \ref{lem:tdelta}, the sought expression for $\Delta_{ijk}(\alpha, \beta)$ can be determined.
\begin{theorem}
    {Assume that $\Pi$ is a valid collocation projection with an associated quadrature strength of $\tau$ and a primary aliasing coefficient $d_{\Pi}$.} Then, if $0 \leq i,j,k \leq p$ and $i+j+k=\tau+2$,
    \begin{equation}
        \Delta_{ijk}(\alpha, \beta)
        =
        \frac{d_{\Pi}}{2}
        \frac{
        \begin{pmatrix}
            2i \\ i
        \end{pmatrix}
        \begin{pmatrix}
            2j \\ j
        \end{pmatrix}
        \begin{pmatrix}
            2k \\ k
        \end{pmatrix}
        }
        {
        \begin{pmatrix}
            2\tau+2 \\ \tau+1
        \end{pmatrix}
        }
        (\beta(i+j) - \alpha(\tau+2-(i+j))).
    \end{equation}
    \label{theo:Delta}
\end{theorem}
\begin{proof}
    The proof follows from direct application of Lemma \ref{lem:tdelta} and the definition of $\Delta_{ijk}(\alpha,\beta)$.
\end{proof}
From Theorem \ref{theo:Delta}, it can be observed that, similar to the behaviour described by \cite{blaisdell1996effect} for spectral methods, the aliasing error for the conservative ($\alpha =1$, $\beta=0$) and non-conservative ($\alpha=0$, $\beta=1$) DG discretizations are opposite in sign. Therefore, split form DG schemes decrease aliasing errors in the weakly underresolved range via a mechanism that is qualitatively similar to that previously observed for split form spectral discretizations. These dealiasing mechanisms are, however, not mathematically identical. In particular, for spectral methods, it can be shown \cite{blaisdell1996effect} that the error on the lowest-order aliased mode vanishes for the average split ($\alpha=1/2$, $\beta=1/2$), hence rendering this split optimal for dealiasing purposes. This is not the case for DG methods, as shown by the following corollary.
\begin{corollary}
    {Assume that $\Pi$ is a valid collocation projection with an associated quadrature strength of $\tau$ and a primary aliasing coefficient $d_{\Pi}$.} If $0 \leq k \leq p$ and $i+j=\tau-p+2$
    \begin{equation}
        \Delta_{ijk}(\alpha, \beta) = 0 \quad \text{for} \quad
        \alpha = \frac{\tau-p+2}{\tau+2}, \quad
        \beta = \frac{p}{\tau+2}.
    \end{equation}
    In particular for the $(p+1)$ GL nodes, this will occur when
    \begin{equation}
        \alpha = \frac{p+3}{2p+3}, \quad
        \beta = \frac{p}{2p+3},
    \end{equation}
    while for the $(p+1)$ GLL nodes, one has
    \begin{equation}
        \alpha = \frac{p+1}{2p+1}, \quad
        \beta = \frac{p}{2p+1}.
    \end{equation}
    \label{corr:optimal_splits}
\end{corollary}
\begin{proof}
    The proof follows directly by application of Theorem \ref{theo:Delta} and the fact that $\tau = 2p+1$ for the GL quadrature and $\tau = 2p-1$ for the GLL quadrature.
\end{proof}
{We can expect the split values introduced in Corollary \ref{corr:optimal_splits} to minimize the volumetric aliasing errors caused by quadratic non-linearities provided that the numerical scheme operates in the well-resolved range, \textit{ie.}, most of the energy associated with the aliased modes is contained in the Legendre flux component of order $i+j=\tau-p+2$. Indeed, by construction, these split coefficients eliminate volumetric aliasing errors on the product $L_iL_j$ for $i+j = \tau-p+2$. Aliasing errors will, however, still be present for higher-order products. Hence, if the mesh used is too coarse to represent adequately solution features, the dealiased split values from Corollary \ref{corr:optimal_splits} will not minimize aliasing error. It should also be stressed that Corollary \ref{corr:optimal_splits} focuses {only on volumetric aliasing errors as part of the scheme's total truncation error.} This is discussed in more details in sections \ref{sec:T_integration} and \ref{sec:inexact_mass}.}

\subsection{Connection with Burgers Entropy Stability}
\label{subsec:E_stab}
We complete our theoretical analysis of split form discretizations by relating and comparing the analysis conducted to the entropy stable split form DG scheme ($\alpha = 2/3$, $\beta=1/3$) for the Burgers equation introduced by \cite{gassner2013skew}. Returning to Eq.(\ref{eq:FD_continuous}) and Eq.(\ref{eq:Delta_def}), we can observe that, in the Legendre basis, split form DG discretizations for the one-dimensional Burgers equation can be written as
\begin{align}
    \int_\Omega \frac{d}{dt}\Pi(L_k u) d\Omega
    - \frac{1}{2}\int_\Omega \frac{d L_k}{dx} u^2 d\Omega
    + \frac{1}{2}\sum_{i,j=0}^pu_iu_j \Delta_{ijk}(\alpha, \beta)
    + \int_{\Gamma} L_k f^* \hat{n} d\Gamma
    = 0.
\end{align}
The associated discrete entropy equation is obtained by using $u$ as the polynomial test function,
\begin{equation}
    \frac{d}{dt}\frac{1}{2}\sum_{k=0} u_k^2 \int_{\Omega}\Pi(L_k^2)d\Omega
    - \frac{1}{2}\int_\Omega \frac{d u}{dx} u^2 d\Omega
    + \frac{1}{2}\sum_{i,j,k=0}^p u_iu_ju_k \Delta_{ijk}(\alpha, \beta)
    + \int_{\Gamma} uf^* \hat{n} d\Gamma
    = 0.
\end{equation}
Applying integration by parts, this is equivalent to
\begin{align}
    \frac{d}{dt}\frac{1}{2}\sum_{k=0} u_k^2 \int_{\Omega}\Pi(L_k^2)d\Omega
    - \int_{\Gamma} \left(\frac{1}{6}u^3 - u f^*\right)\hat{n} d\Gamma
    &= -\frac{1}{2}\sum_{i,j,k=0}^p u_iu_ju_k \Delta_{ijk}(\alpha, \beta) \\
    &= -\frac{1}{6}\sum_{i,j,k=0}^p u_iu_ju_k (\Delta_{ijk}(\alpha, \beta) + \Delta_{kij}(\alpha, \beta) + \Delta_{kji}(\alpha, \beta)),
    \label{eq:Estab_Delta}
\end{align}
where symmetry has been used in the final step. For local entropy stability (assuming that $f^*$ is entropy stable in the sense of Tadmor \cite{tadmor1987numerical}), the right-hand side of Eq.(\ref{eq:Estab_Delta}) must vanish no matter the value of $u_i u_j u_k$. By definition, we note that
\begin{equation}
    \Delta_{ijk} + \Delta_{kij} + \Delta_{kji}
    =
    (1-3\beta)\int_\Gamma L_iL_jL_k\hat{n}d\Gamma
    + (2\beta-\alpha)\int_\Omega\left(
    \frac{d L_i}{dx} \Pi(L_j L_k)
    +
    \frac{d L_j}{dx} \Pi(L_i L_k)+\frac{d L_k}{dx} \Pi(L_i L_j)
    \right)d\Omega.
\end{equation}
Thus, the tensor $\Delta_{ijk}(\alpha,\beta)$ will satisfy the skew-symmetry relation
\begin{equation}
    \Delta_{ijk} + \Delta_{kij} + \Delta_{kji} = 0
\end{equation}
if and only if
\begin{equation}
    1-3\beta = 0 \quad \text{and} \quad 2\beta - \alpha = 0
    \iff
    \alpha = 2/3 \quad \text{and} \quad \beta = 1/3,
\end{equation}
in which case the DG scheme will be rendered entropy stable. Hence, we can see that while the split coefficients from Corollary \ref{corr:optimal_splits} ensure that the aliasing error is minimized in the well-resolved range by eliminating the error on the lowest order aliased modes, the choice $\alpha = 2/3$, $\beta = 1/3$ guarantees that aliasing errors do not contaminate the total entropy of the solution by forcing their sum over all modes to vanish. Alternatively, we can say that the choice of splitting coefficients affects the structural properties of the error tensor $\Delta_{ijk}(\alpha, \beta)$; using the split coefficients from Corollary \ref{corr:optimal_splits} requires the diagonal characterized by $i+j+k=\tau+2$ to vanish and using $\alpha = 2/3$, $\beta = 1/3$ renders the error tensor skew-symmetric.

\subsection{Impact of Temporal Integration}
\label{sec:T_integration}
{While, as mentioned previously, the split values introduced in Corollary \ref{corr:optimal_splits} will minimize volumetric aliasing errors associated with the split form semi-discretization, this does not denote that they will minimize the global truncation error on the numerical solution when the scheme is integrated in time. {Indeed, for the one-dimensional Burgers problem, for sufficiently large integration times, there is no theoretical guarantee preventing aliasing errors on the higher-order Legendre products from triggering unbounded growth of the solution's energy.} As discussed in the previous section, this issue is specifically handled by the entropy stable split ($\alpha = 2/3$, $\beta= 1/3$), which equips $\Delta_{ijk}(\alpha, \beta)$ with the structural properties required to bound the numerical solution's energy when the discretization is endowed with an entropy stable surface numerical flux. It is worth noting that although the entropy stable split is not constructed to eliminate aliasing errors on low-order product terms, it will nevertheless achieve dealiasing in the weakly underresolved range by virtue of its split form structure, as shown by Theorem \ref{theo:Delta}. It should be stressed that while this brief section has been concerned with temporal integration of split form DG discretizations, the conclusions discussed are properties of the associated semi-discretizations and will hence hold regardless of the choice of discrete time-stepping scheme.}

\subsection{Effects of Inexact Mass Matrix}
\label{sec:inexact_mass}
{Up to now, our analysis has focused on aliasing errors introduced by the integration of the volumetric component of split form discretizations. As mentioned previously, if the mass matrix is inexact, additional aliasing errors will be generated through the integration of the temporal derivative of the numerical solution. Although split form discretizations will still retain their dealiasing capabilities in this case, the coefficient values from Corollary \ref{corr:optimal_splits} will no longer result in perfect cancellation of the aliasing error of low-order Legendre polynomial products. In fact, careful analysis of this case shows that no such optimal dealiased split can be derived, as the total aliasing error is dependent on the individual values of $i$ and $j$ --instead of depending only on $i+j$ as in Theorem \ref{theo:Delta}.}

\section{Numerical Experiments}
\label{sec:num_exp}
{We validate our theoretical findings by conducting different numerical experiments on the scheme given by Eq.(\ref{eq:FD_formulation}) and Eq.(\ref{eq:2ptflux}) using $(p+1)$ GL and $(p+1)$ GLL quadrature nodes. For experiments utilizing the GLL nodes, we consider schemes with exact and inexact mass matrices. As using an exact -- and hence non-diagonal -- mass matrix with GLL quadrature nodes significantly reduces the practical benefits of the latter, this choice of scheme construction should be regarded more as a useful theoretical tool to systematically investigate volumetric aliasing errors.}

\subsection{Error Tensor Visualization}
In the previous section, we showed that the aliasing error for split form DG discretizations of quadratic non-linearities could be described by the tensor $\Delta_{ijk}(\alpha,\beta)$. Additionally, we demonstrated how different choices of splitting coefficients resulted in the latter featuring distinct structural properties. In this section, we complete this analysis by plotting all non-zero entries of $\Delta_{ijk}(\alpha, \beta)$ for the case $p=7$ with the $(p+1)$ GL nodes using different pairs of splitting coefficients. {As previously discussed, all non-zero entries of $\Delta_{ijk}(\alpha,\beta)$ occur for indices satisfying $i+j+k \geq \tau + 2$ and the tensor inherits the staggered structure of $\Pi$. Hence, non-zero entries of $\Delta_{ijk}(\alpha, \beta)$ are confined to a tetrahedral region in $ijk$ space, as exemplified in Figure \ref{fig:tensor_visual_side} for the conservative form ($\alpha=1$, $\beta=0$). In Figure \ref{fig:tensor_visual}, we plot an isometric view of $\Delta_{ijk}(\alpha,\beta)$ for different values of splitting coefficients. The colour of the markers is used to represent the value of the non-zero entries of the tensor, while the size of the latter has been chosen to encode depth along the direction of the isometric projection. In particular, all markers of a given size are located on the plane $i+j+k=N$ where $N$ is some arbitrary integer. The smallest marker size is associated to the case $i+j+k=\tau+2$, which corresponds to all entries of $\Delta_{ijk}(\alpha,\beta)$ that can be computed via Theorem \ref{theo:Delta}.}

As expected from Theorem \ref{theo:Delta}, the conservative form exhibits the largest entries on the diagonal $i+j+k=\tau+2$ followed by the average split and the entropy stable split. All entries associated with the diagonal $i+j+k=\tau+2$ vanish for the dealiased split introduced in Corollary \ref{corr:optimal_splits}. As mentioned previously, we expect the integration error in the weakly underresolved range to be predominantly governed by entries of $\Delta_{ijk}(\alpha, \beta)$ on the diagonal $i+j+k=\tau+2$.

The skew-symmetric structure of $\Delta_{ijk}$ for the entropy stable split can also be observed in Figure \ref{fig:tensor_visual}. It can be seen that while this split is not optimized for dealiasing in the sense of Corollary \ref{corr:optimal_splits}, it still significantly reduces the magnitude of the entries of $\Delta_{ijk}$ over a wide range of modes when compared to the conservative form.

From Figure \ref{fig:tensor_visual}, it is clear that split forms do not achieve dealiasing by merely reducing the integration error over all Legendre modes; they rather allow one to fine-tune the modes over which this type of error will be significant. This is a fundamental difference compared to dealiasing techniques relying on overintegration. In fact, none of the choices of splitting coefficients considered in Figure \ref{fig:tensor_visual} can be deemed superior over all indices $ijk$. For instance, $|\Delta_{ppp}|$ is lower for the conservative form than for the average split and the dealiased split from Corollary \ref{corr:optimal_splits} while the entropy stable split exhibits a higher aliasing error than the dealiased split on the diagonal $i+j+k=\tau+2$.

{While, for brevity, we omit plotting the tensor entries for the $(p+1)$ GLL integration nodes, results are qualitatively similar to those presented in Figure \ref{fig:tensor_visual} and discussed in this section.}
\begin{figure}
    \centering
    \includegraphics[width=0.5\linewidth]{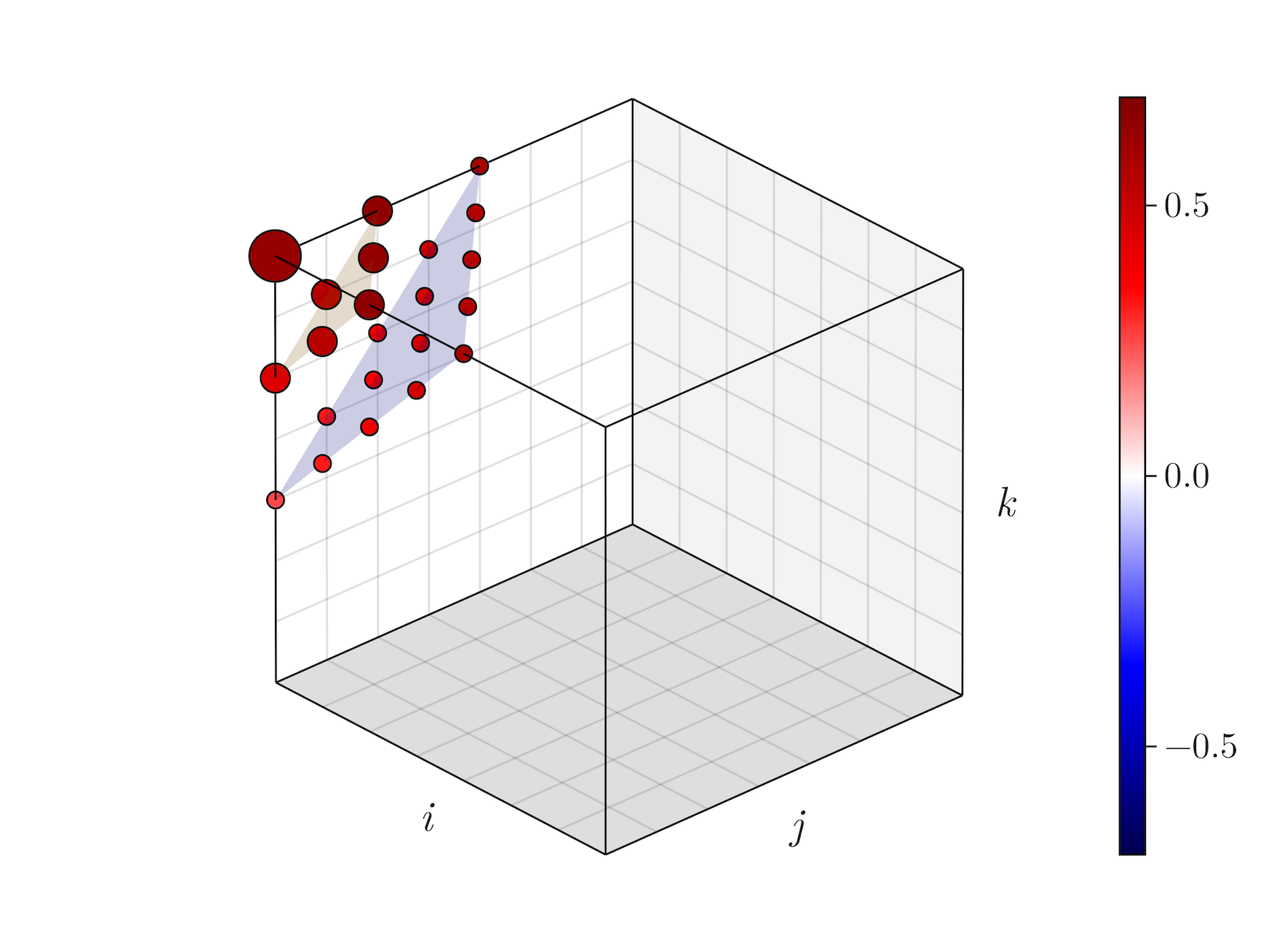}
    \caption{Side view of the tensor $\Delta_{ijk}(1,0)$ in the Legendre basis using $(p+1)$ GL quadrature nodes ($p=q=7$). Planes associated with constant values of $i+j+k$ have been coloured to emphasize the tetrahedral structure of the non-zero entries of $\Delta_{ijk}(\alpha, \beta)$.}
    \label{fig:tensor_visual_side}
\end{figure}
\begin{figure}
    \centering
    \begin{subfigure}{0.49\textwidth}
        \centering
        \includegraphics[width=\linewidth]{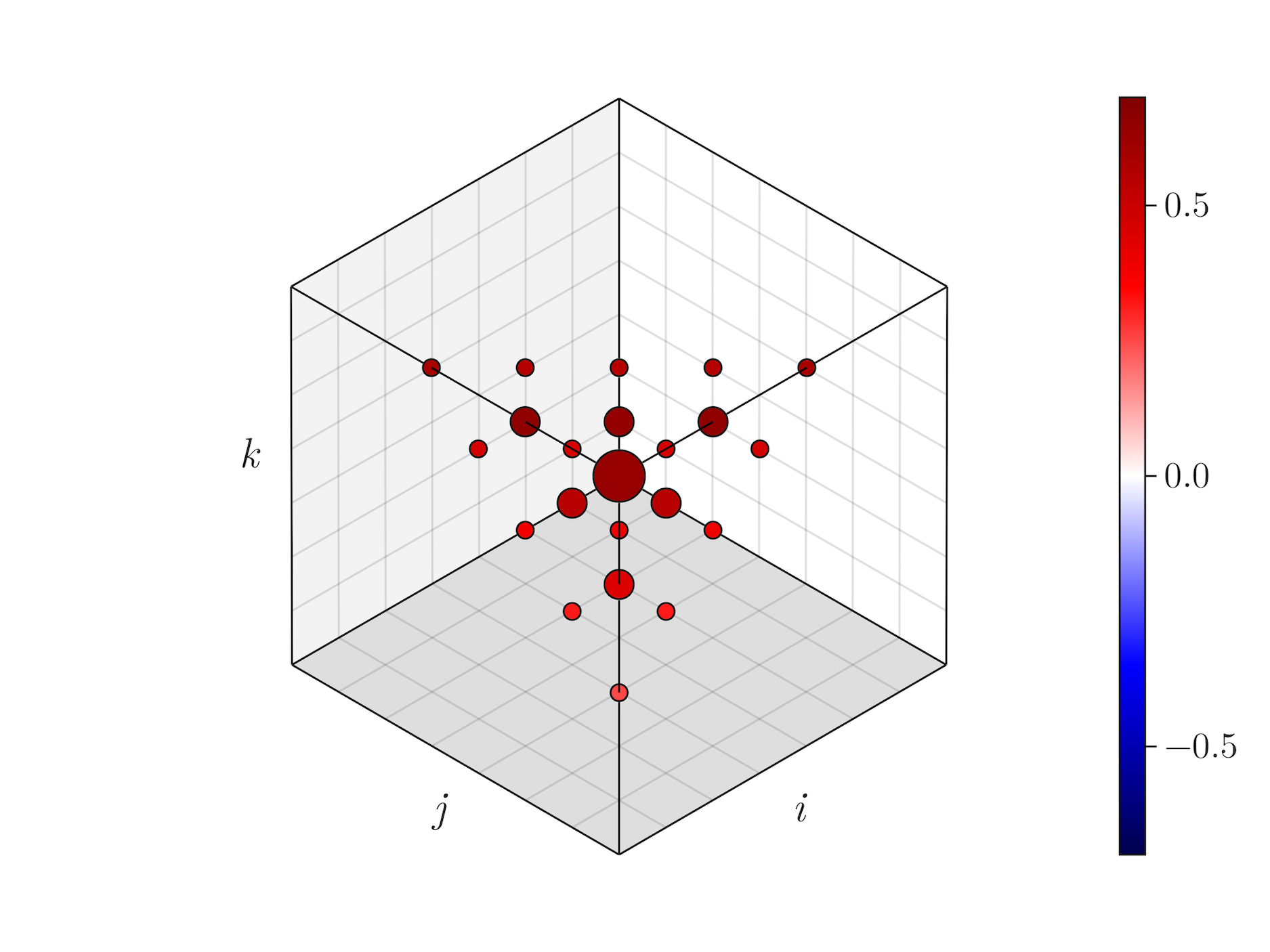}
        \caption{Conservative ($\alpha=1$, $\beta=0$).}
    \end{subfigure}
    \hfill
    \begin{subfigure}{0.49\textwidth}
        \centering
        \includegraphics[width=\linewidth]{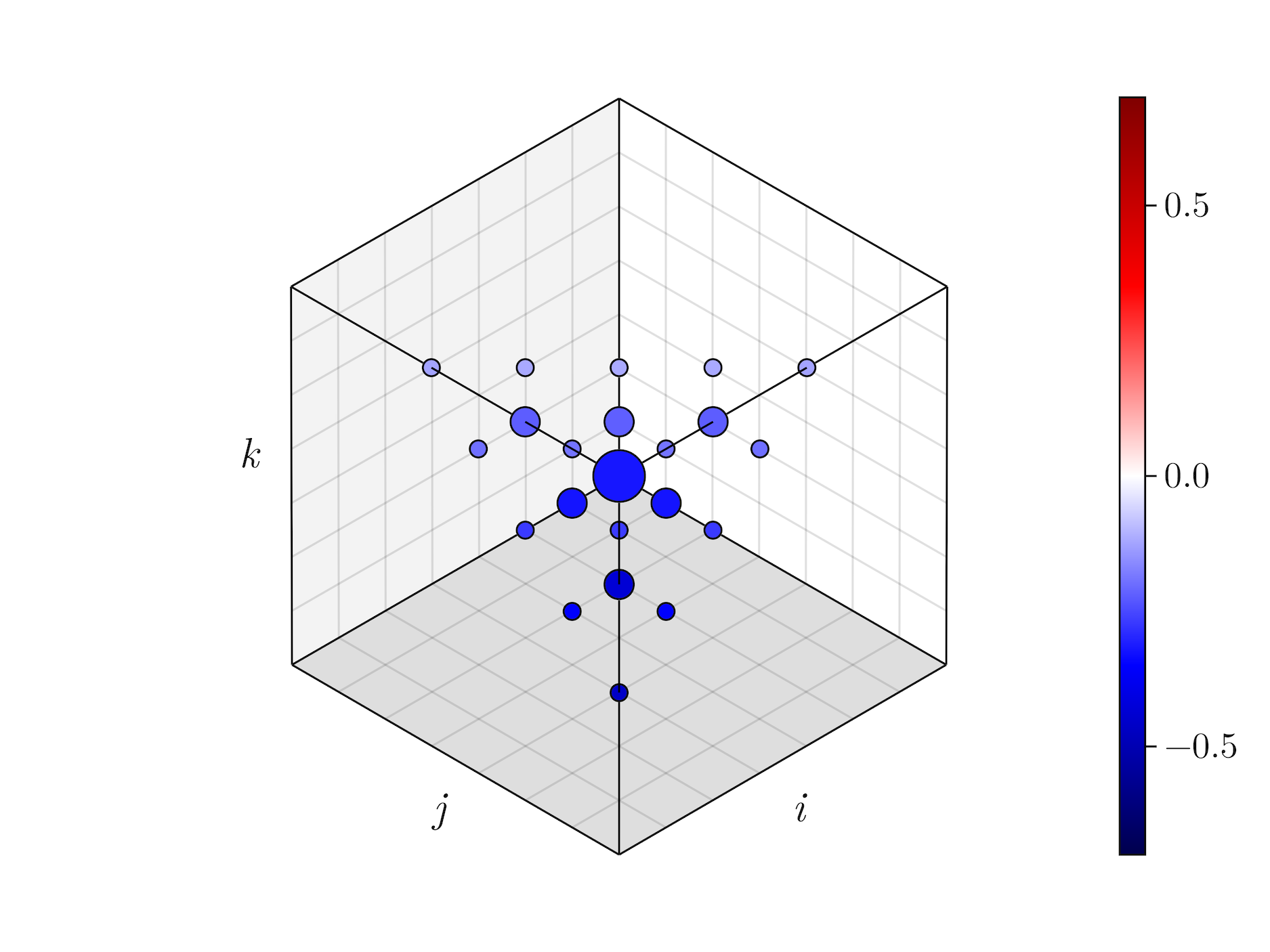}
        \caption{Average ($\alpha=1/2$, $\beta=1/2$).}
    \end{subfigure}
    \centering
    \begin{subfigure}{0.49\textwidth}
        \centering
        \includegraphics[width=\linewidth]{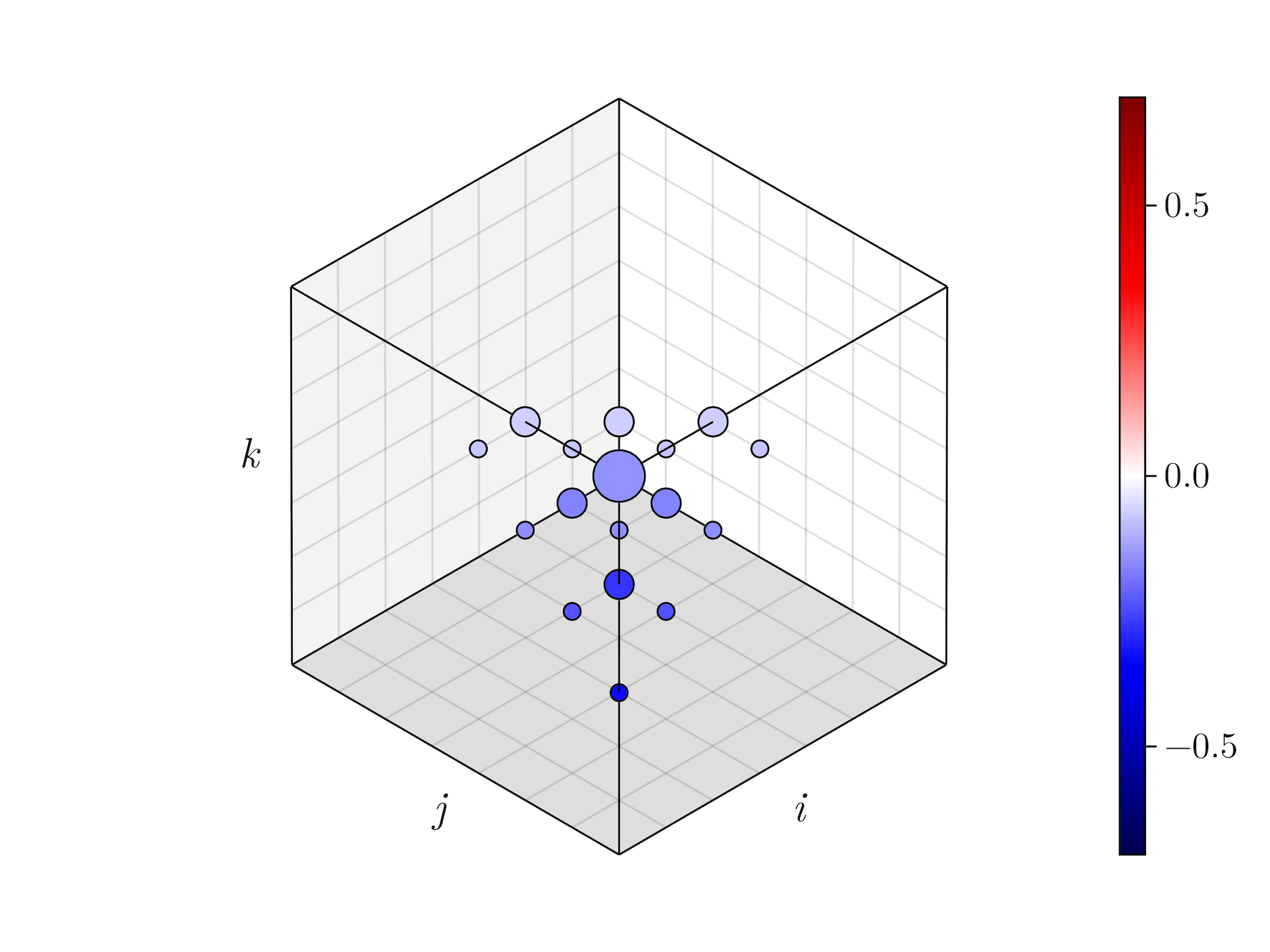}
        \caption{Dealiased ($\alpha=(p+3)/(2p+3)$, $\beta=p/(2p+3)$).}
    \end{subfigure}
    \hfill
    \begin{subfigure}{0.49\textwidth}
        \centering
        \includegraphics[width=\linewidth]{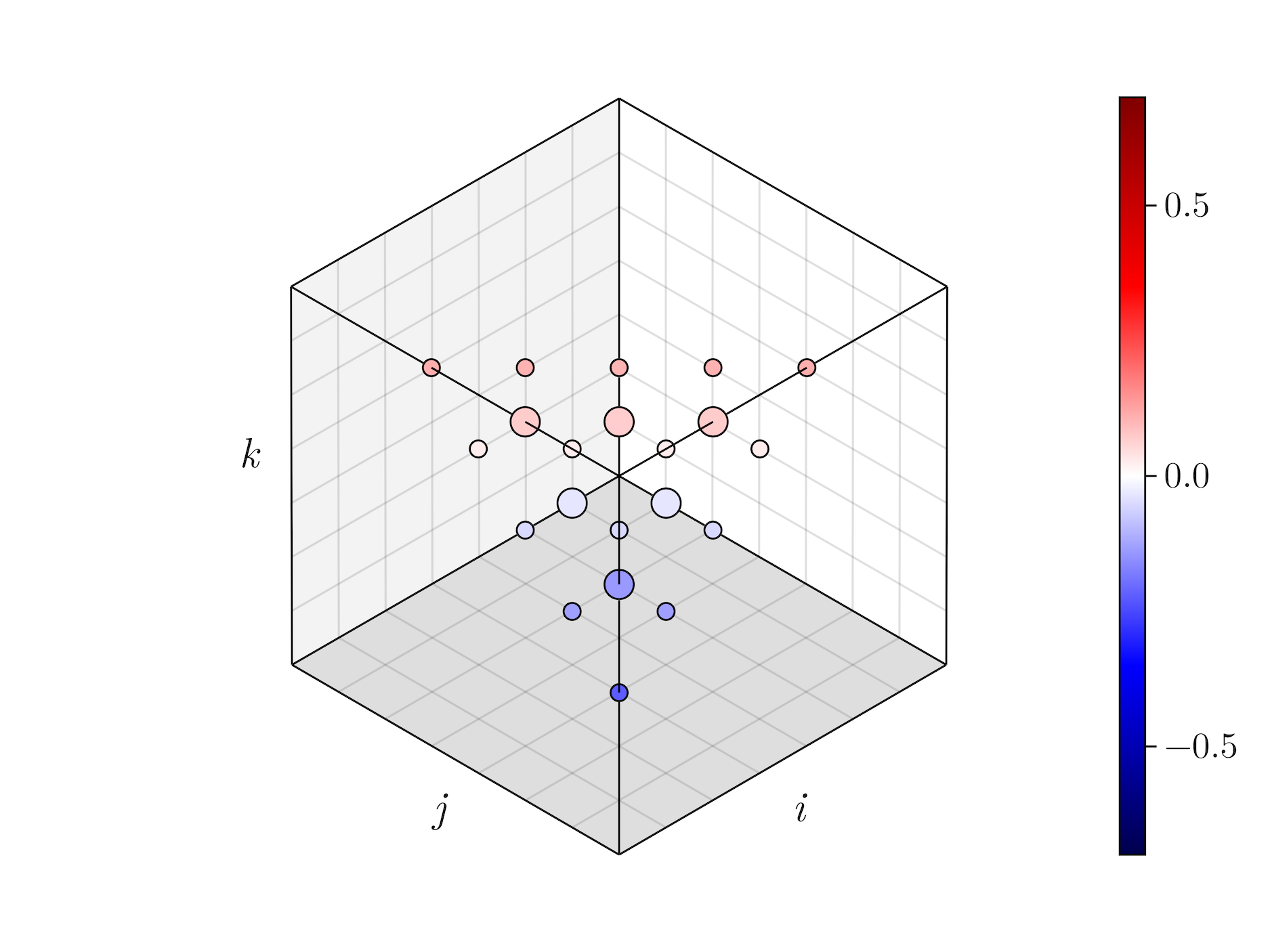}
        \caption{Entropy Stable ($\alpha=2/3$, $\beta=1/3$).}
    \end{subfigure}
    \caption{Structure of the tensor $\Delta_{ijk}(\alpha,\beta)$ in the Legendre basis using $(p+1)$ GL quadrature nodes ($p=q=7$).}
    \label{fig:tensor_visual}
\end{figure}

\subsection{Error for Split Form DG Discretizations}
{We first compare the error associated with different split form DG discretizations for the 1D Burgers problem. To do so, we consider the initial conditions
\begin{equation}
    u(x,0) = \sin(2\pi (x-0.1)) + 0.01,
    \label{eq:sin_IC}
\end{equation}
and evolve the system on the periodic domain $[-0.5, 0.5]$ with uniform elements of size $\Delta x$ using the split forms of interest and a Lax-Friedrich numerical flux at the elemental interfaces. We integrate in time until $t = 0.3/\pi$ utilizing a sufficiently small time step $\Delta t$ to render temporal discretization effects irrelevant. It should be noted that this test case is very similar to that considered by \cite{gassner2013skew}.}

\subsubsection{L2 Error}
\label{sec:L2_error}
In Figures \ref{fig:L2_error_GL}, \ref{fig:L2_error_GLL} and \ref{fig:L2_error_GLL_EXACT}, we plot the L2 error associated with the numerical solution as a function of the number of degrees of freedom for $p=5$ and $p=8$ split form discretizations. Figure \ref{fig:L2_error_GL} shows results obtained for the $(p+1)$ GL quadrature nodes, while Figures \ref{fig:L2_error_GLL} and \ref{fig:L2_error_GLL_EXACT} were generated for the $(p+1)$ GLL quadrature nodes using exact and inexact mass matrices. {In all cases, error levels are compared to those associated with a DG scheme using a volume quadrature of sufficient strength to exactly integrate the quadratic non-linearity.}

{As can be observed, all schemes achieve an approximate empirical convergence rate of $O(\Delta x^{p+1})$. From Figure \ref{fig:L2_error_GL}, it can be seen that the L2 error levels associated with split forms using GL quadrature nodes are essentially indistinguishable from the exactly integrated case.} Since the strength of the GL quadrature is $\tau = 2p+1$, as shown in \ref{app4}, the integration error must scale with $O(\Delta x^{p+2})$ for smooth solutions. Hence, aliasing effects are masked from the global truncation error results reported in Figure \ref{fig:L2_error_GL}. As will be shown in section \ref{sec:results_stability}, this does not, however, render aliasing errors irrelevant as they are inherently tied to scheme stability.

{Since the GLL quadrature is associated with a strength of $\tau=2p-1$, following the argument presented in \ref{app4}, the integration error must scale with $O(\Delta x^{p})$ and hence the choice of splitting coefficients is expected to affect the global truncation error measured. This is consistent with the results reported in Figures \ref{fig:L2_error_GLL} and \ref{fig:L2_error_GLL_EXACT}. For schemes that utilize an inexact mass matrix, split forms are associated with marginally, but consistently, smaller error levels than the conservative form in the asymptotic regime. Additionally, in this case, the dealiased split is marginally outperformed by the average split due to additional aliasing errors introduced by inexact integration of the temporal derivative. These results are consistent with the discussion presented in section \ref{sec:inexact_mass}. For the schemes using an exact mass matrix, Figure \ref{fig:L2_error_GLL_EXACT} reveals a clear difference in the error levels in the asymptotic range. In particular, the dealiased split is associated with the lowest error, followed by the average split and the entropy stable split. In this case, the dealiased split and the exactly integrated scheme are almost undistinguishable. The conservative form features the highest error levels. These results are consistent with the theoretical predictions stemming from Theorem \ref{theo:Delta} and Corollary \ref{corr:optimal_splits}.}
\begin{figure}
    \centering
    \begin{subfigure}{0.49\textwidth}
        \centering
        \includegraphics[width=\linewidth]{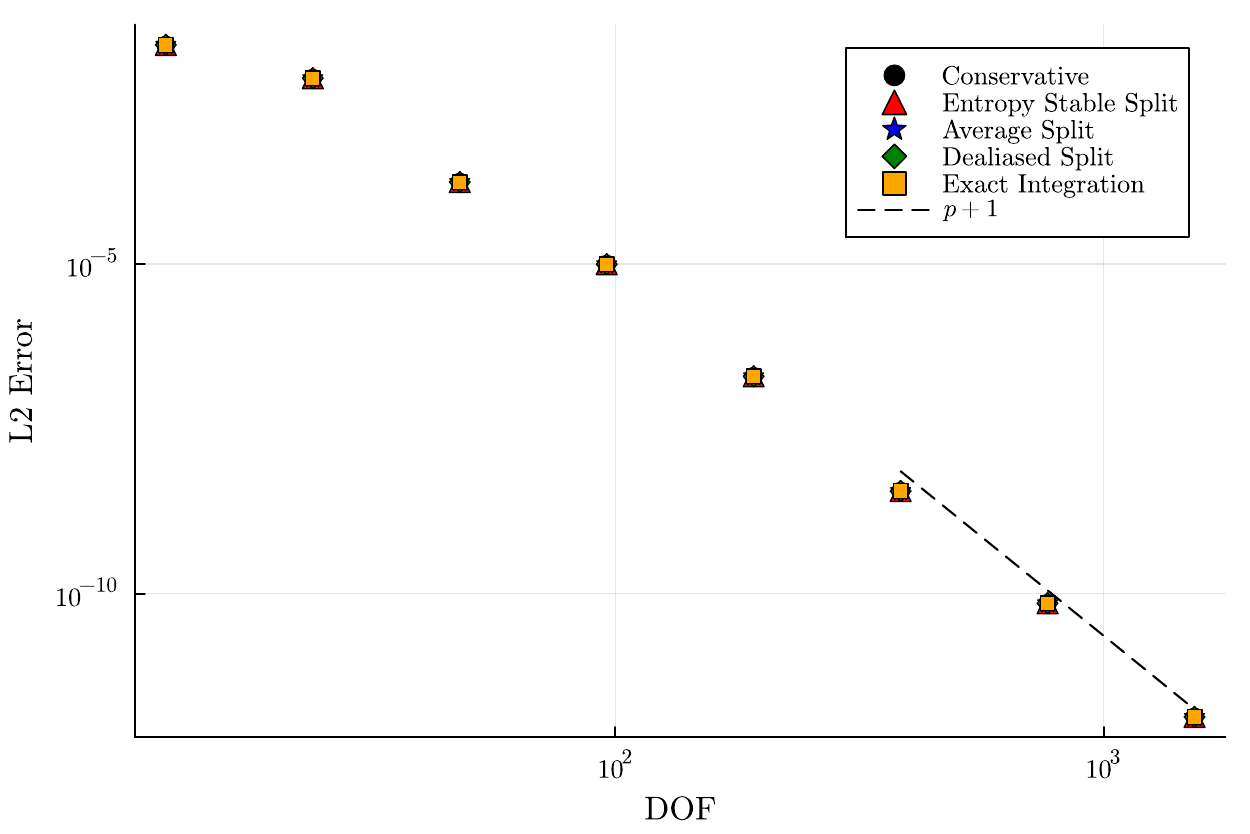}
        \caption{$p=5$.}
    \end{subfigure}
    \hfill
    \begin{subfigure}{0.49\textwidth}
        \centering
        \includegraphics[width=\linewidth]{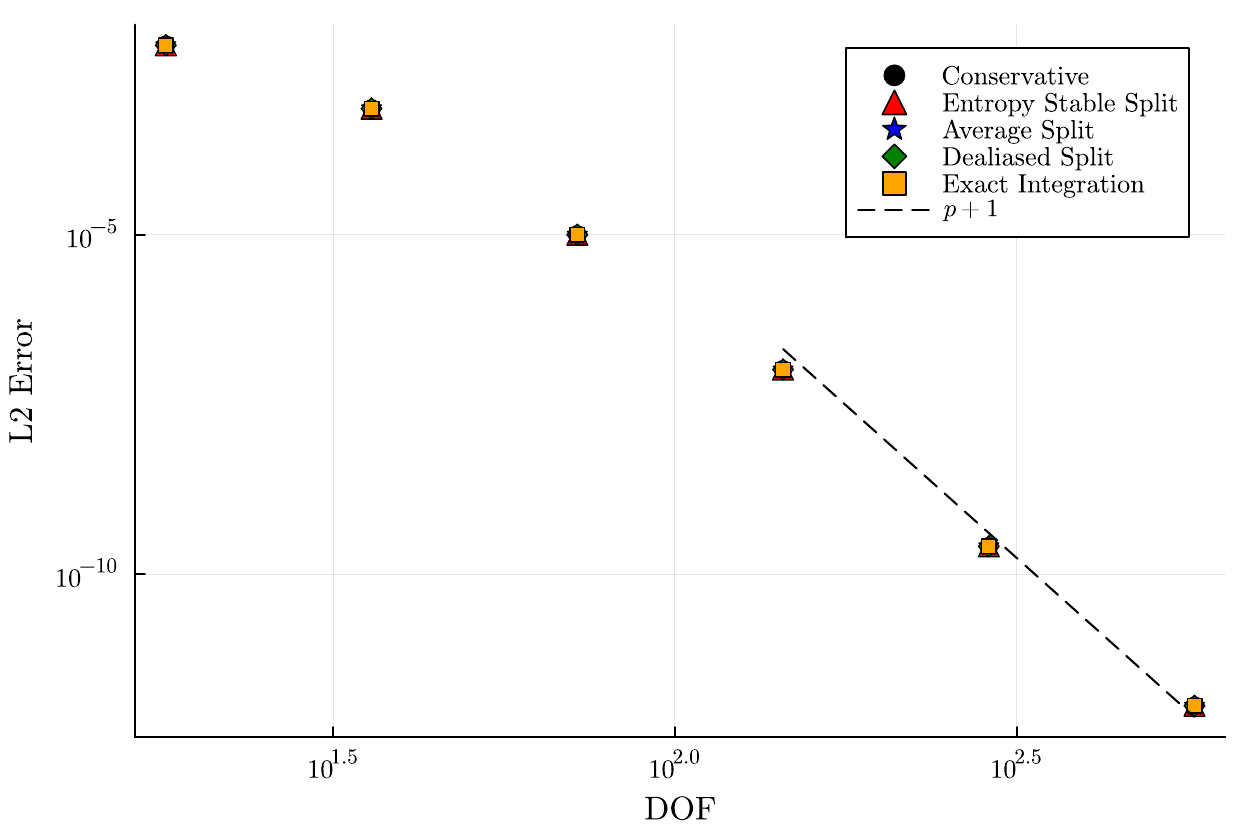}
        \caption{$p=8$.}
    \end{subfigure}
    \caption{L2 error for the different split forms using $(p+1)$ GL quadrature nodes. L2 error levels for all split forms approximately overlap with the exactly integrated DG scheme.}
    \label{fig:L2_error_GL}
\end{figure}
\begin{figure}
    \centering
    \begin{subfigure}{0.49\textwidth}
        \centering
        \includegraphics[width=\linewidth]{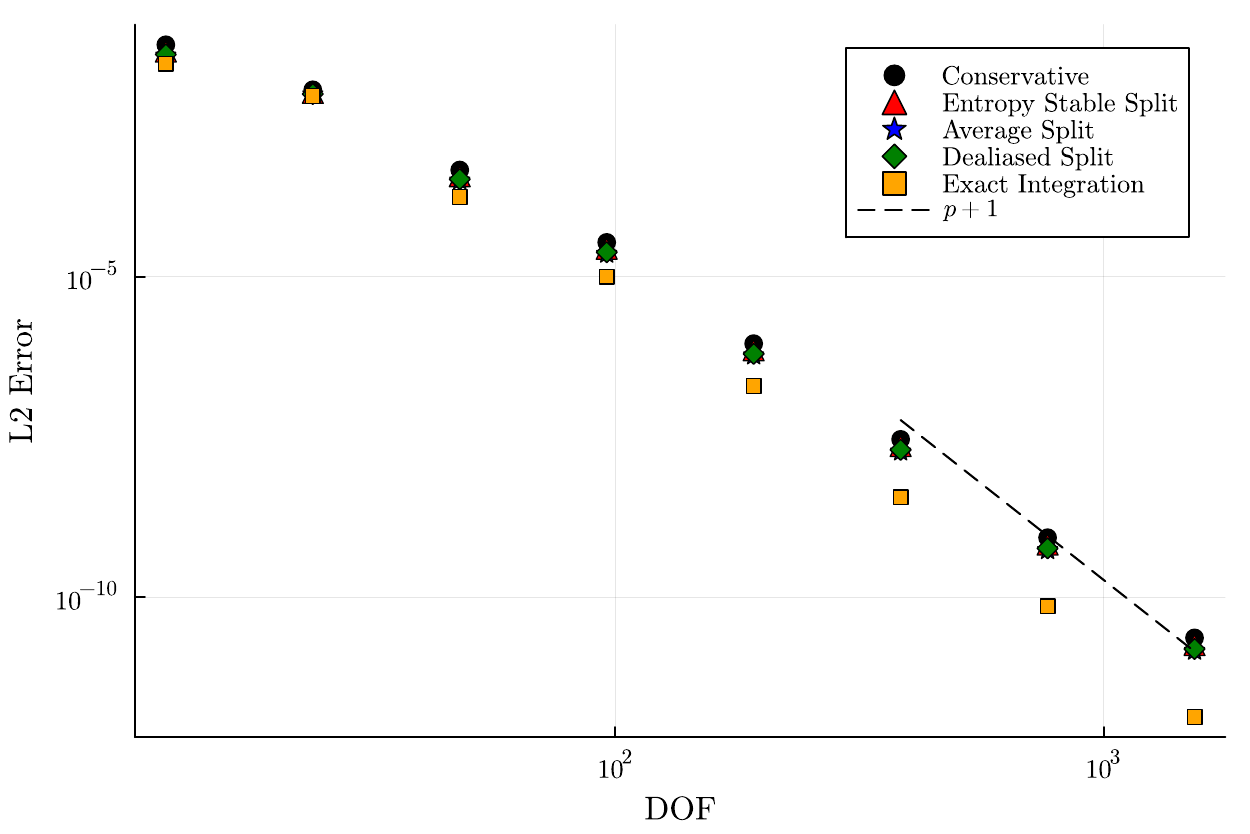}
        \caption{$p=5$.}
    \end{subfigure}
    \hfill
    \begin{subfigure}{0.49\textwidth}
        \centering
        \includegraphics[width=\linewidth]{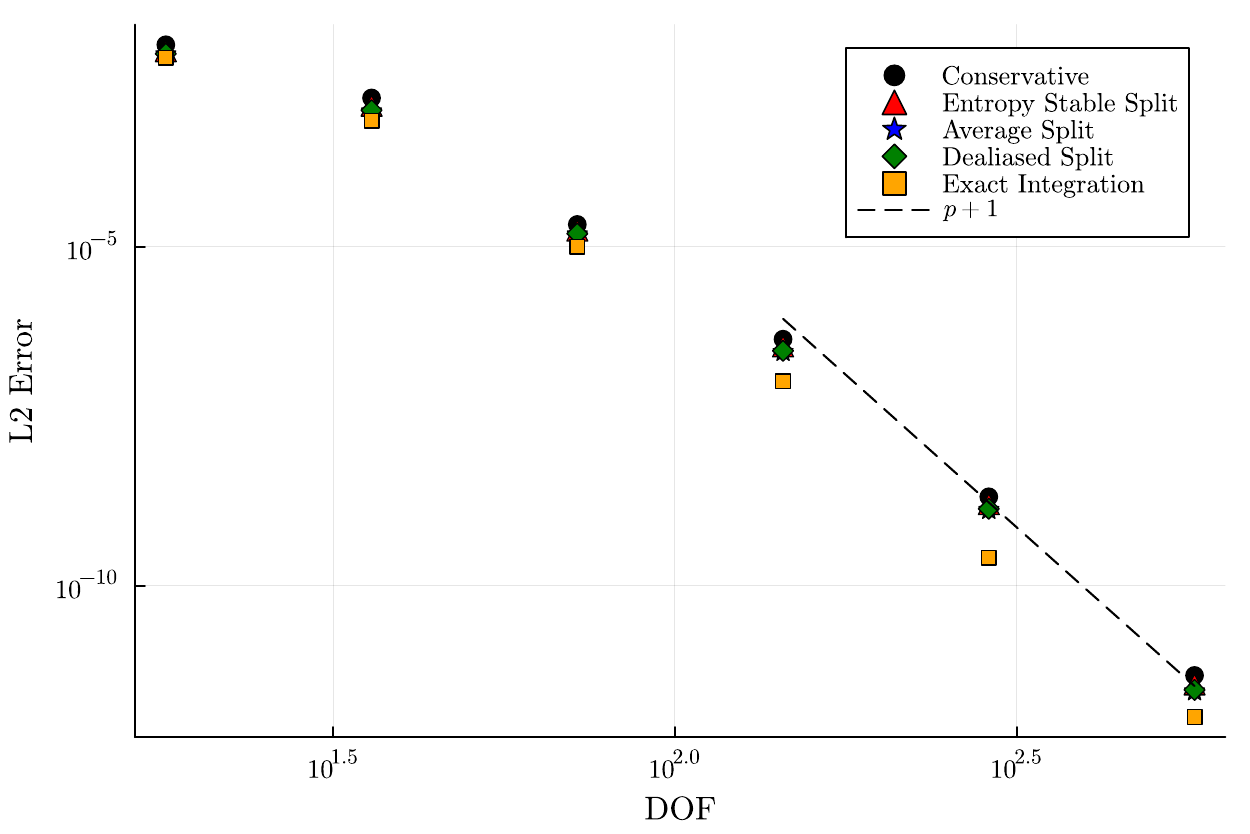}
        \caption{$p=8$.}
    \end{subfigure}
    \caption{L2 error for the different split forms using $(p+1)$ GLL quadrature nodes and an inexact mass matrix. L2 error levels for all split forms approximately overlap.}
    \label{fig:L2_error_GLL}
\end{figure}
\begin{figure}
    \centering
    \begin{subfigure}{0.49\textwidth}
        \centering
        \includegraphics[width=\linewidth]{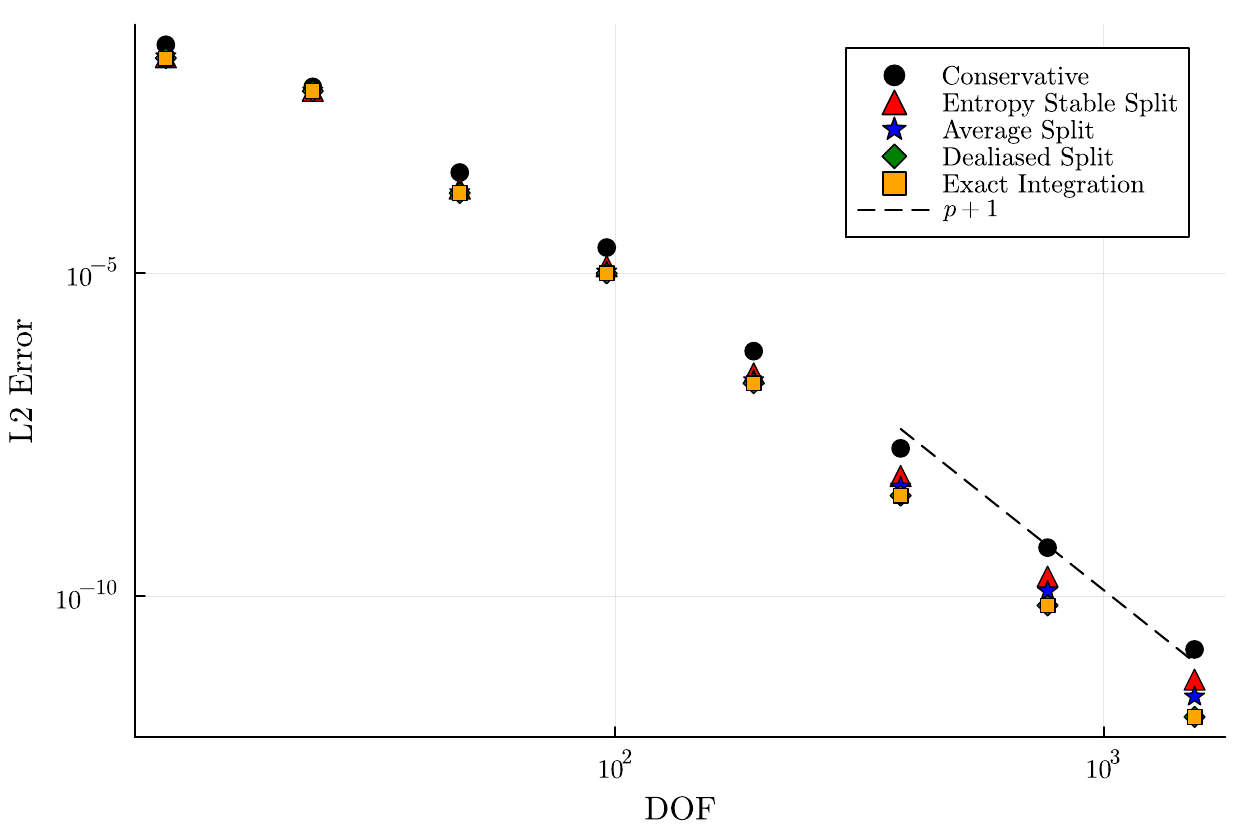}
        \caption{$p=5$.}
    \end{subfigure}
    \hfill
    \begin{subfigure}{0.49\textwidth}
        \centering
        \includegraphics[width=\linewidth]{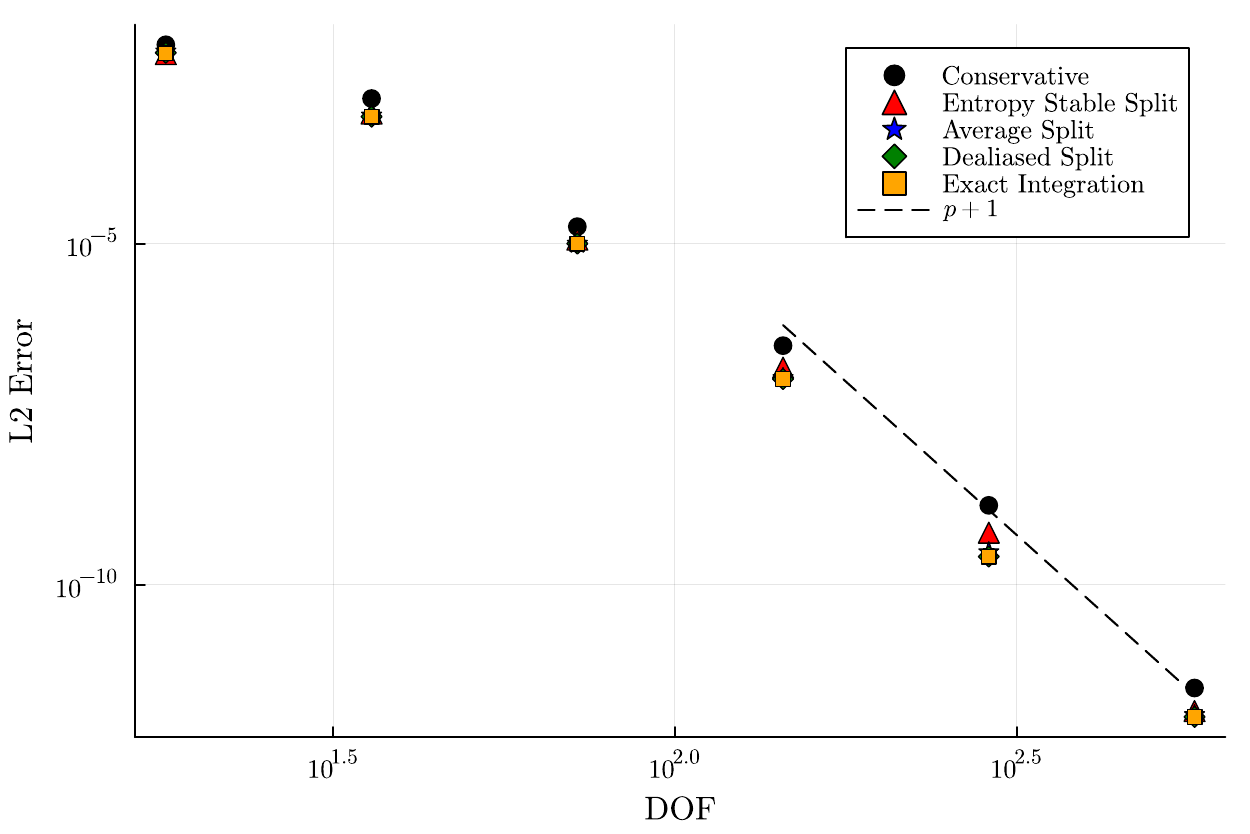}
        \caption{$p=8$.}
    \end{subfigure}
    \caption{L2 error for the different split forms using $(p+1)$ GLL quadrature nodes and an exact mass matrix. L2 error levels for the dealiased split and the exactly integrated DG scheme approximately overlap.}
    \label{fig:L2_error_GLL_EXACT}
\end{figure}

\subsubsection{Integration Error}
{To further our analysis, we consider isolating the integration error component from the global truncation error studied in the previous section. To do so, we consider the same test case as in section \ref{sec:L2_error} and consider computing the $L2$ error between the numerical solution obtained via the different split forms and an exactly integrated DG scheme. For sufficiently small integration times, this process allows us to estimate the integration error empirically.}

{The integration error for different split forms as a function of the number of degrees of freedom is shown in Figures \ref{fig:integration_error_GL}, \ref{fig:integration_error_GLL} and \ref{fig:integration_error_GLL_EXACT} for $p=5$ and $p=8$ schemes. As can be seen in Figures \ref{fig:integration_error_GL} and \ref{fig:integration_error_GLL_EXACT}, for split forms utilizing the GL quadrature or the GLL quadrature with an exact mass matrix, the conservative form, the average split and the entropy stable split result in an integration error scaling asymptotically with $O(\Delta x^{\tau-p+1})$. The dealiased splits from Corollary \ref{corr:optimal_splits} are, however, associated with a higher asymptotic convergence rate of $O(\Delta x^{\tau-p+2})$. This is consistent with the expected orders for the integration error (\ref{app4}) and the fact that the dealiased split is defined to eliminate all integration errors associated with the lowest-order aliased mode. Moreover, it can be seen that the relative integration error levels in Figures \ref{fig:integration_error_GL} and \ref{fig:integration_error_GLL_EXACT} are correctly predicted by Theorem \ref{theo:Delta}. In particular, for the case $p=5$ with GL quadrature nodes, the dealiased split is associated with the lowest integration error, followed by the entropy stable split and the average split. The conservative form results in the highest integration error. For the case $p=8$ with GL quadrature nodes and the cases $p=5$ and $p=8$ with GLL quadrature nodes and an exact mass matrix, results are similar, but the average split yields a lower integration error than the entropy stable split.}

{Finally, as can be observed in Figure \ref{fig:integration_error_GLL}, when GLL quadrature nodes are used with an inexact mass matrix, the integration error for the conservative form and all split formulations scales asymptotically with $O(\Delta x^p)$. As the integration error is now dominated by aliasing effects caused by inexactness of the mass matrix, the difference in error levels between the different split forms is significantly reduced. Nevertheless, split forms are still associated with a lower integration error than the conservative form in this case. As previously touched upon, the dealiased split is no longer associated with the smallest integration error.}
\begin{figure}
    \centering
    \begin{subfigure}{0.49\textwidth}
        \centering
        \includegraphics[width=\linewidth]{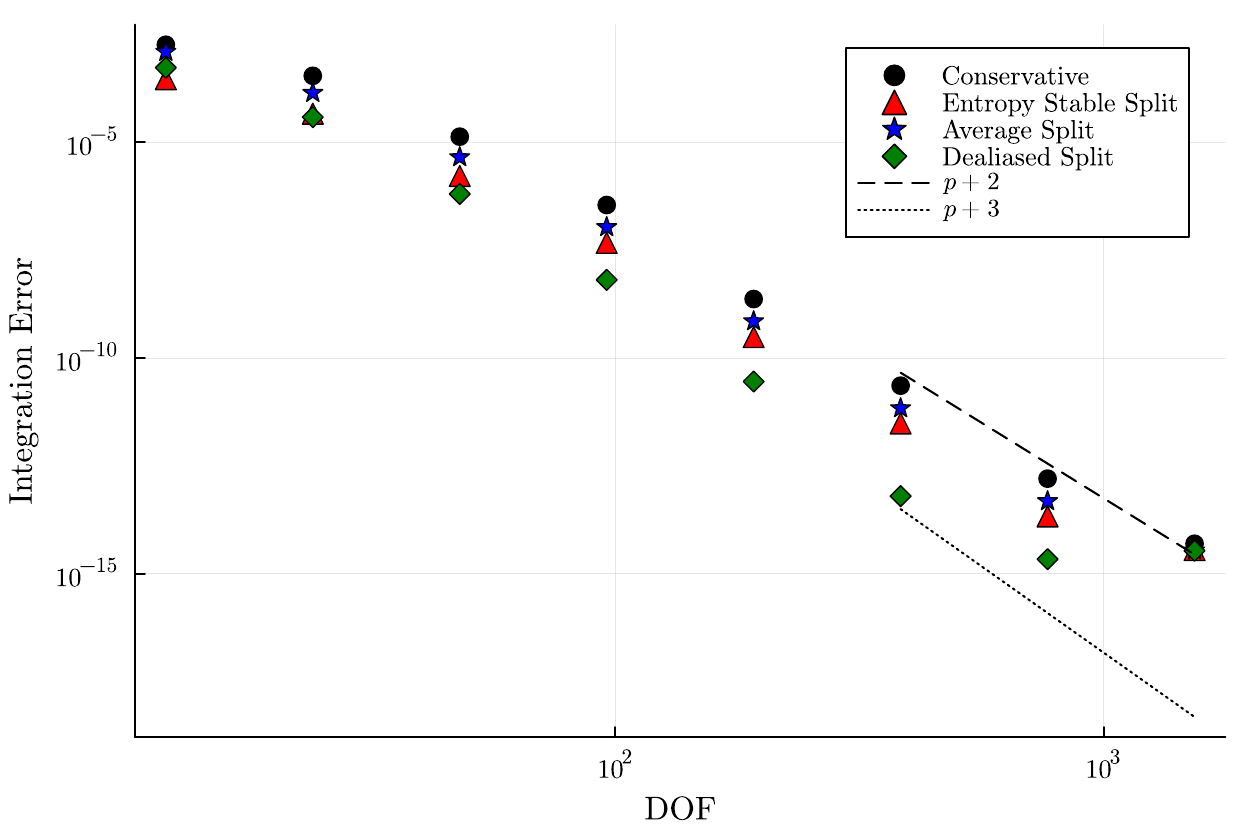}
        \caption{$p=5$.}
    \end{subfigure}
    \hfill
    \begin{subfigure}{0.49\textwidth}
        \centering
        \includegraphics[width=\linewidth]{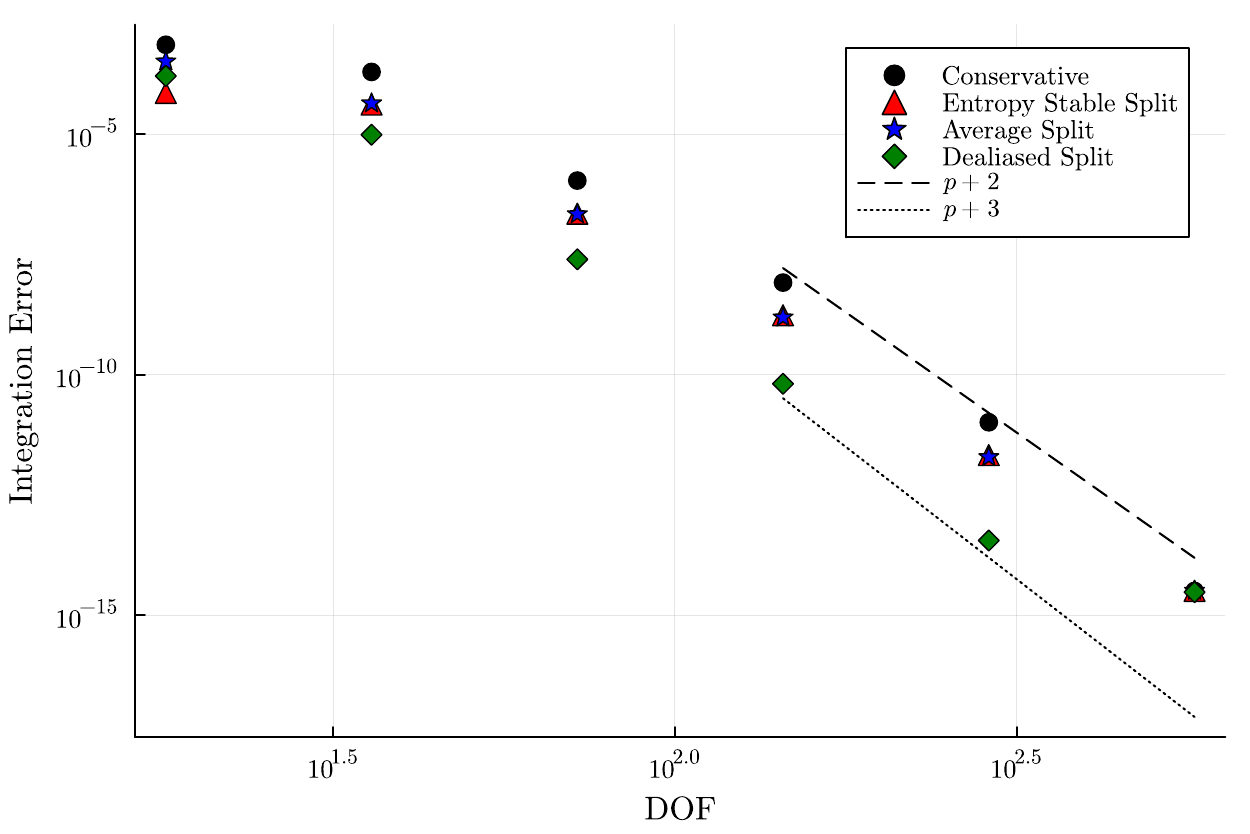}
        \caption{$p=8$.}
    \end{subfigure}
    \caption{Integration error for the different split forms using $(p+1)$ GL quadrature nodes.}
    \label{fig:integration_error_GL}
\end{figure}
\begin{figure}
    \centering
    \begin{subfigure}{0.49\textwidth}
        \centering
        \includegraphics[width=\linewidth]{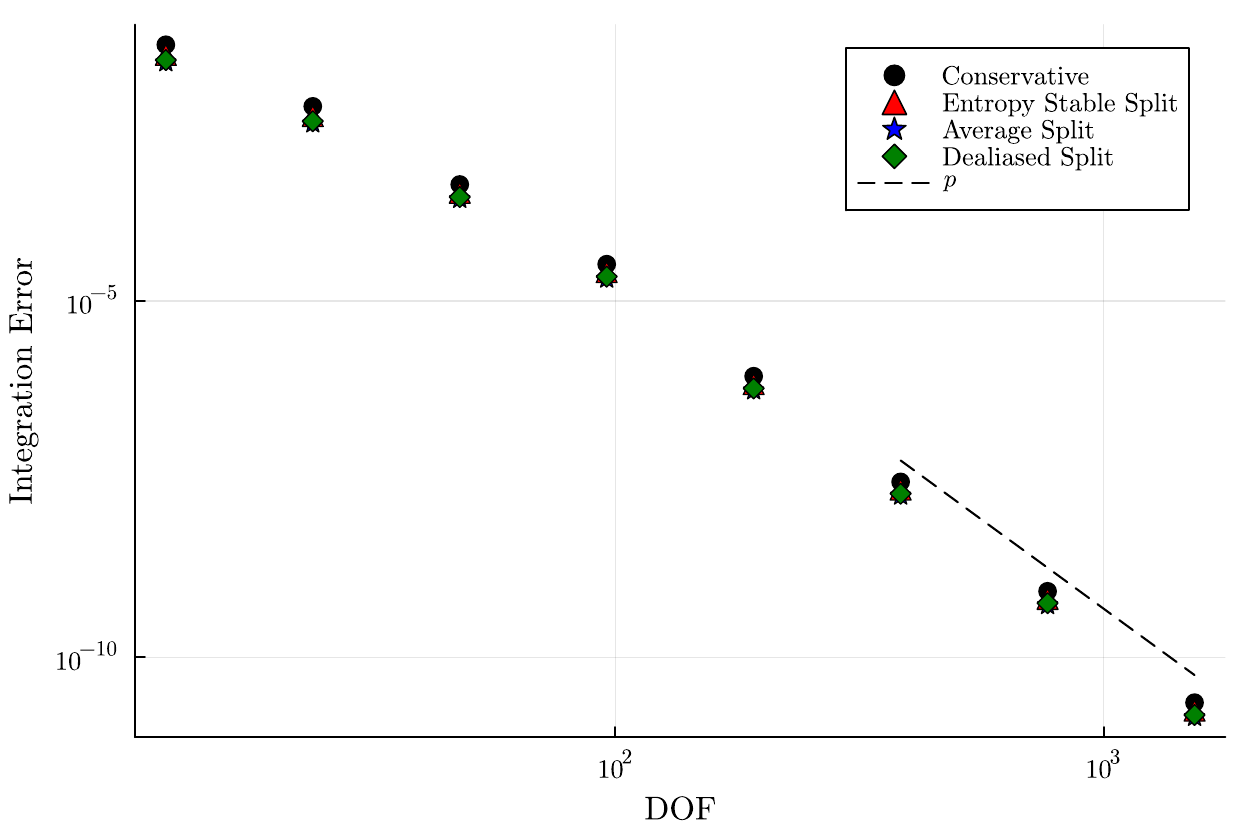}
        \caption{$p=5$.}
    \end{subfigure}
    \hfill
    \begin{subfigure}{0.49\textwidth}
        \centering
        \includegraphics[width=\linewidth]{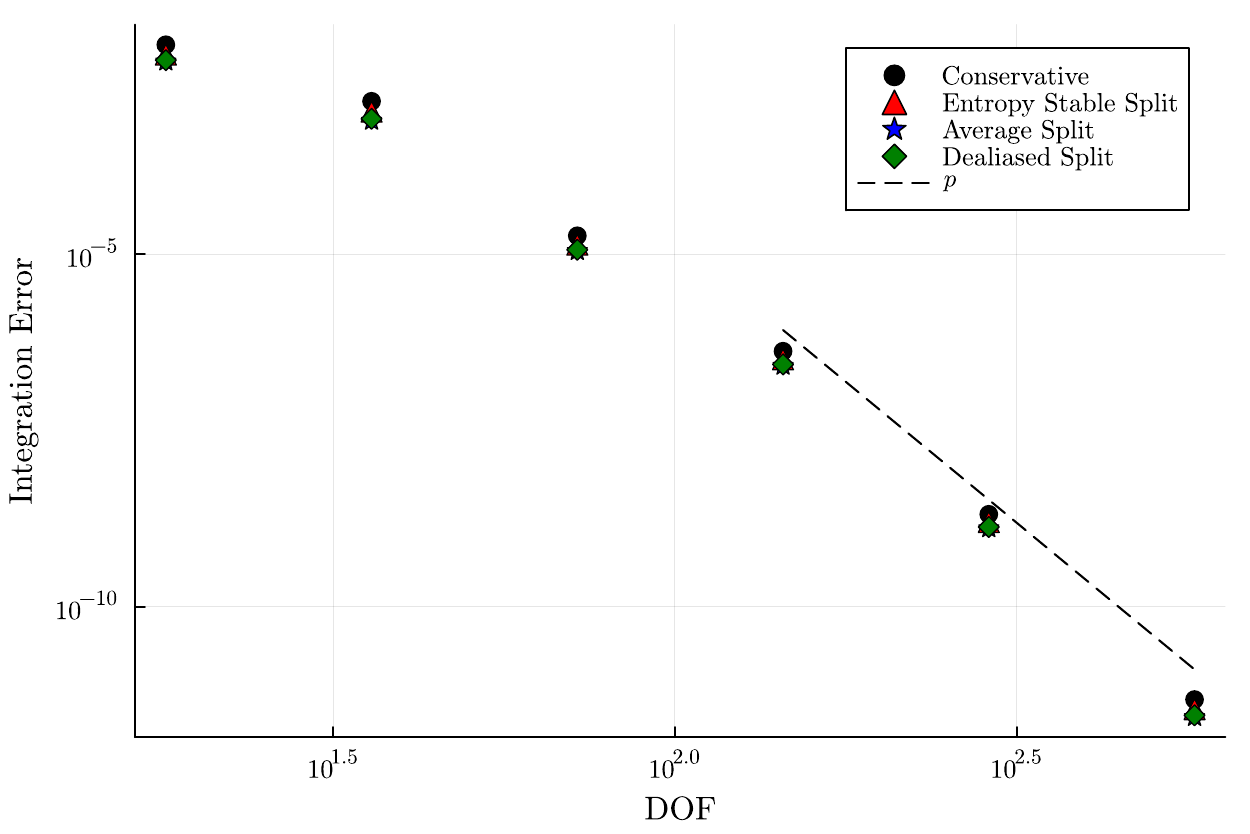}
        \caption{$p=8$.}
    \end{subfigure}
    \caption{Integration error for the different split forms using $(p+1)$ GLL quadrature nodes and an inexact mass matrix. Integration error levels for all split forms approximately overlap.}
    \label{fig:integration_error_GLL}
\end{figure}
\begin{figure}
    \centering
    \begin{subfigure}{0.49\textwidth}
        \centering
        \includegraphics[width=\linewidth]{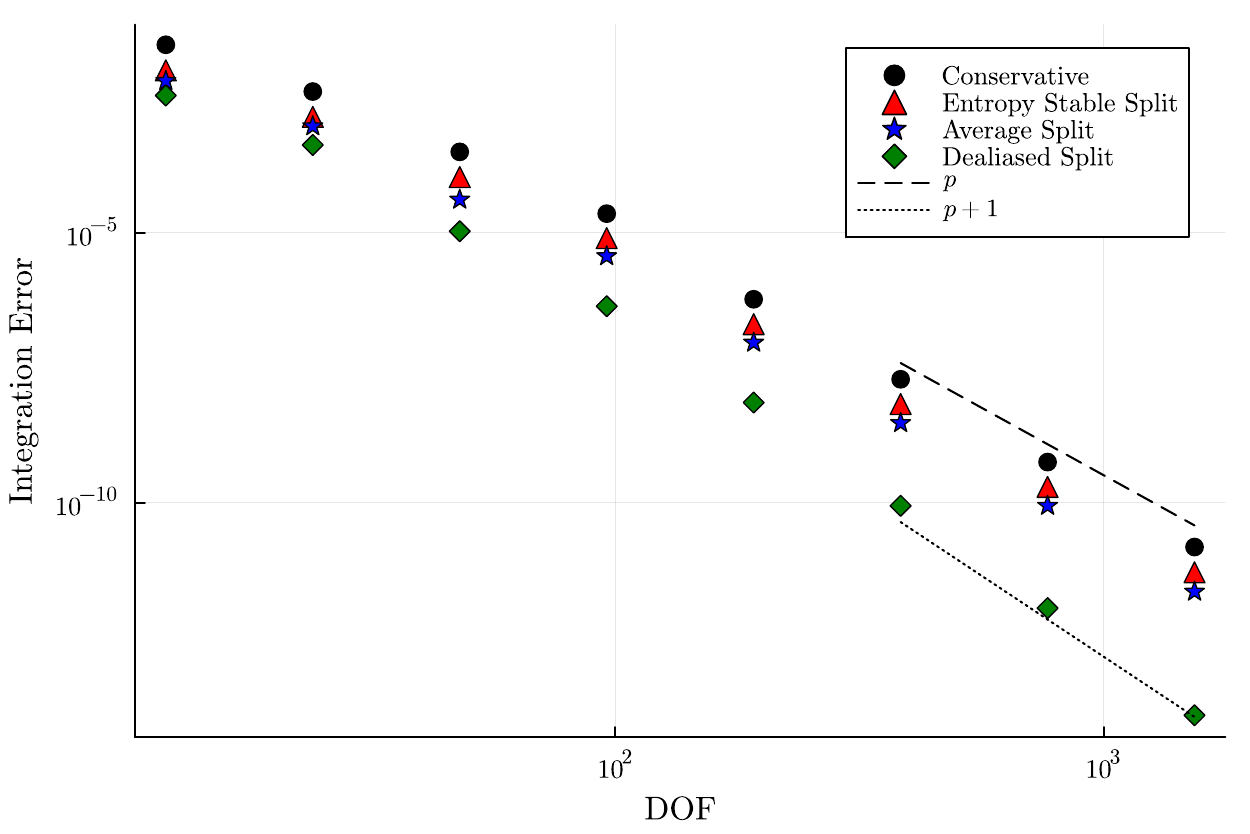}
        \caption{$p=5$.}
    \end{subfigure}
    \hfill
    \begin{subfigure}{0.49\textwidth}
        \centering
        \includegraphics[width=\linewidth]{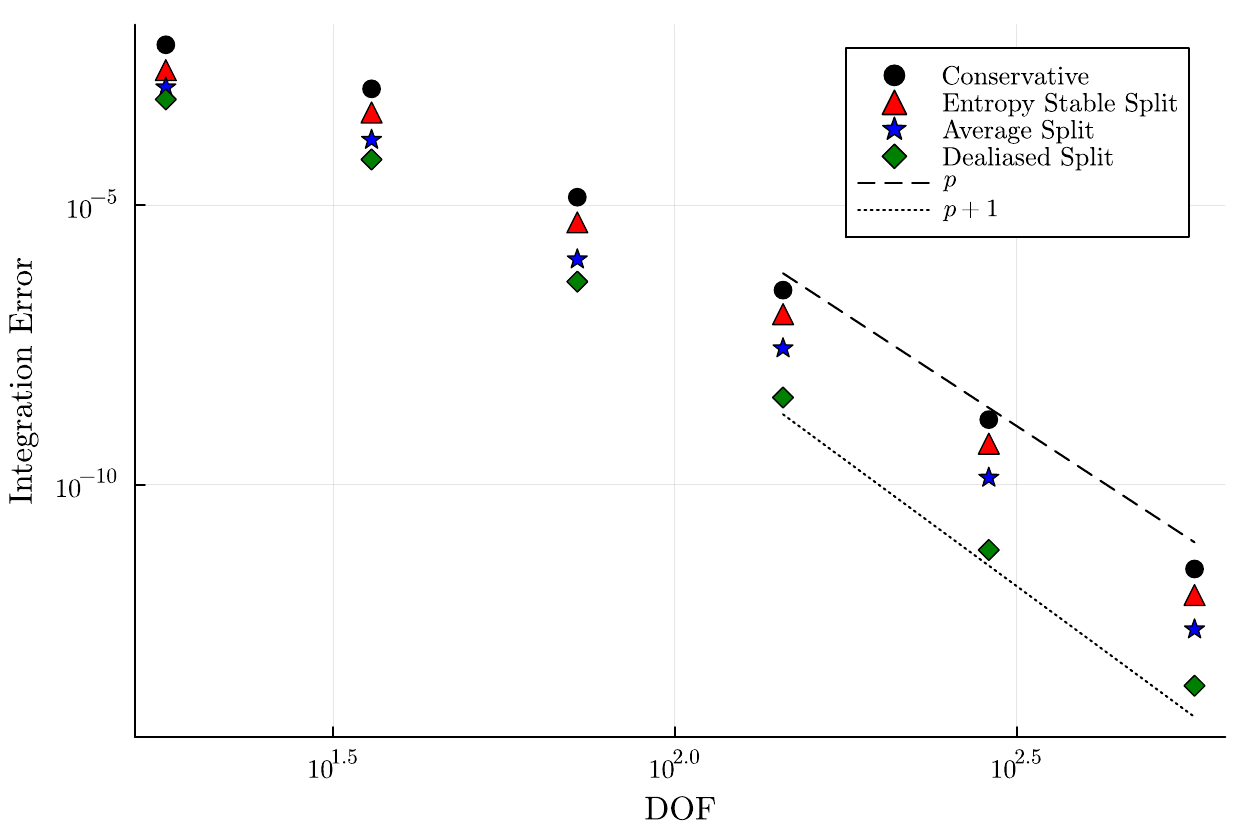}
        \caption{$p=8$.}
    \end{subfigure}
    \caption{Integration error for the different split forms using $(p+1)$ GLL quadrature nodes and an exact mass matrix.}
    \label{fig:integration_error_GLL_EXACT}
\end{figure}

\subsection{Stability of Split Form DG Discretizations}
\label{sec:results_stability}
We finally investigate the temporal stability of split form DG spatial discretizations. To do so, we initialize our system using the initial conditions from Eq.(\ref{eq:sin_IC}) on the periodic domain $[-0.5,0.5]$ discretized using a uniform mesh with 20 elements. We advance the solution in time until $t=3.0$ using the three stage third order strongly stability preserving relaxed Runge-Kutta scheme \cite{ranocha2020relaxation} with 10000 time steps. {As done previously, we also compare the split forms of interest with an exactly integrated DG scheme. By construction, this latter spatial discretization is entropy stable.} In all cases, the Lax Friedrich flux is used at the elemental interfaces.

The time evolution of the quadratic entropy of the numerical solutions obtained with the different split forms is shown in Figure \ref{fig:entropy_vs_time_GL} for $p=5$ and $p=8$ schemes using a GL quadrature. The associated numerical solutions at time $t=3.0$ are shown in Figure \ref{fig:sol_t3} and compared to a reference solution obtained via a finite volume discretization equipped with 2000 elements. As expected, the conservative form is rendered unstable shortly after shock formation, while quadratic entropy is a monotonically decreasing function of time for the entropy stable split. The average split is capable of delaying the solution blow-up in both cases, but the solution is eventually corrupted with aliasing instabilities and the discretization becomes unstable. For the case $p=5$, the numerical solution obtained via the dealiased split remains stable, and the temporal evolution of its associated entropy is almost indistinguishable from that of the entropy stable split. This can be explained by the fact that out of the 11 Legendre modes of the quadratic flux function, only 3 are subject to aliasing errors (as opposed to 4 for the conservative form and the average split). In this case, the instabilities introduced by the comparatively smaller aliasing error appear to be compensated by the numerical dissipation of the scheme. However, when the order of the scheme is increased to $p=8$, the elimination of the lowest-order aliased mode by the dealiased split is insufficient to control the aliasing instabilities. While the entropy still does not blow up in this case, the solution at $t=3.0$ is severely corrupted as shown in Figure \ref{fig:sol_t3}. For a scheme of order $p=9$, the dealiased split was found to lead to numerical solution blow-up for $t \in [0, 3]$.

{The same numerical experiment was also conducted with split forms using the GLL quadrature with inexact and exact mass matrices. The time evolution of the quadratic entropy for $p=5$ and $p=8$ schemes is shown in Figure \ref{fig:entropy_vs_time_GLL} and Figure \ref{fig:entropy_vs_time_GLLEXACT} for the inexact mass matrix and exact mass matrix cases, respectively. As can be seen, as a result of the larger aliasing errors resulting from the lower strength of the GLL quadrature, only the entropy stable split is capable of preserving the stability of the numerical solution. For $p=5$ schemes, all split forms are observed to delay numerical solution blow-up compared to the conservative form, although instabilities in the GLL split forms manifest themselves at earlier times than for their GL counterparts. The close examination of $p=8$ schemes reveals an interesting behaviour. While the conservative form is rendered unstable at a slightly earlier time for the exact mass matrix case, the blow-up of its associated numerical solution can be delayed by equipping the scheme with an inexact mass matrix. In this specific case, the conservative form actually becomes unstable after the average split and the dealiased split. This behaviour can be explained by the fact that utilizing an inexact mass matrix results in the effective filtering of the $p$th Legendre component of the DG residual, thus equipping the scheme with an additional stabilization mechanism. Hence, when aliasing errors are large, the temporal growth of the numerical solution's entropy cannot be understood only through the entries of $\Delta_{ijk}$; the interactions between $\Delta_{ijp}$ and this filtering effect also become significant.}
\begin{figure}
    \centering
    \begin{subfigure}{0.49\textwidth}
        \centering
        \includegraphics[width=\linewidth]{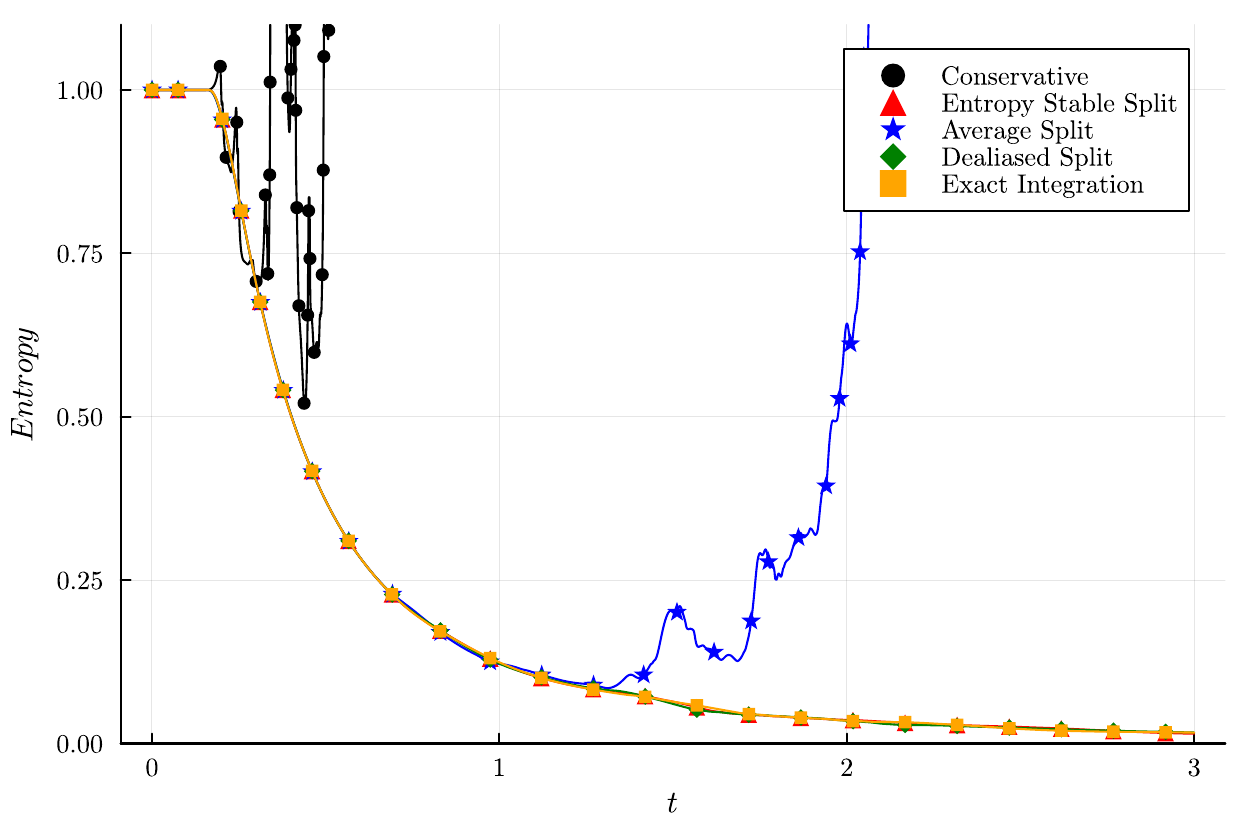}
        \caption{$p=5$.}
    \end{subfigure}
    \hfill
    \begin{subfigure}{0.49\textwidth}
        \centering
        \includegraphics[width=\linewidth]{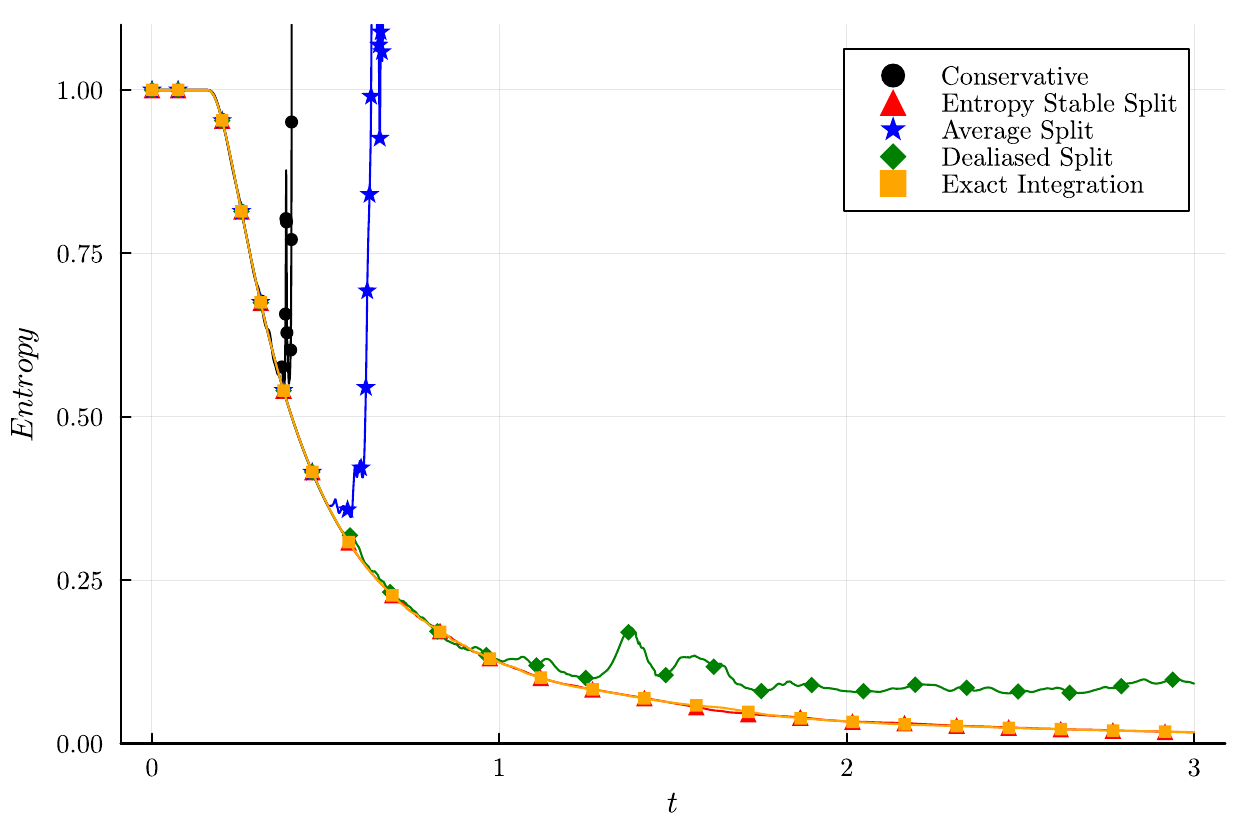}
        \caption{$p=8$.}
    \end{subfigure}
    \caption{Time evolution of the quadratic entropy for the different split forms using $(p+1)$ GL quadrature nodes. The entropy for the exactly integrated scheme, the entropy stable split and the dealiased split approximately overlap for the case $p=5$. The entropy for the exactly integrated scheme and the entropy stable split approximately overlap for the case $p=8$.}
    \label{fig:entropy_vs_time_GL}
\end{figure}
\begin{figure}
    \centering
    \begin{subfigure}{0.49\textwidth}
        \centering
        \includegraphics[width=\linewidth]{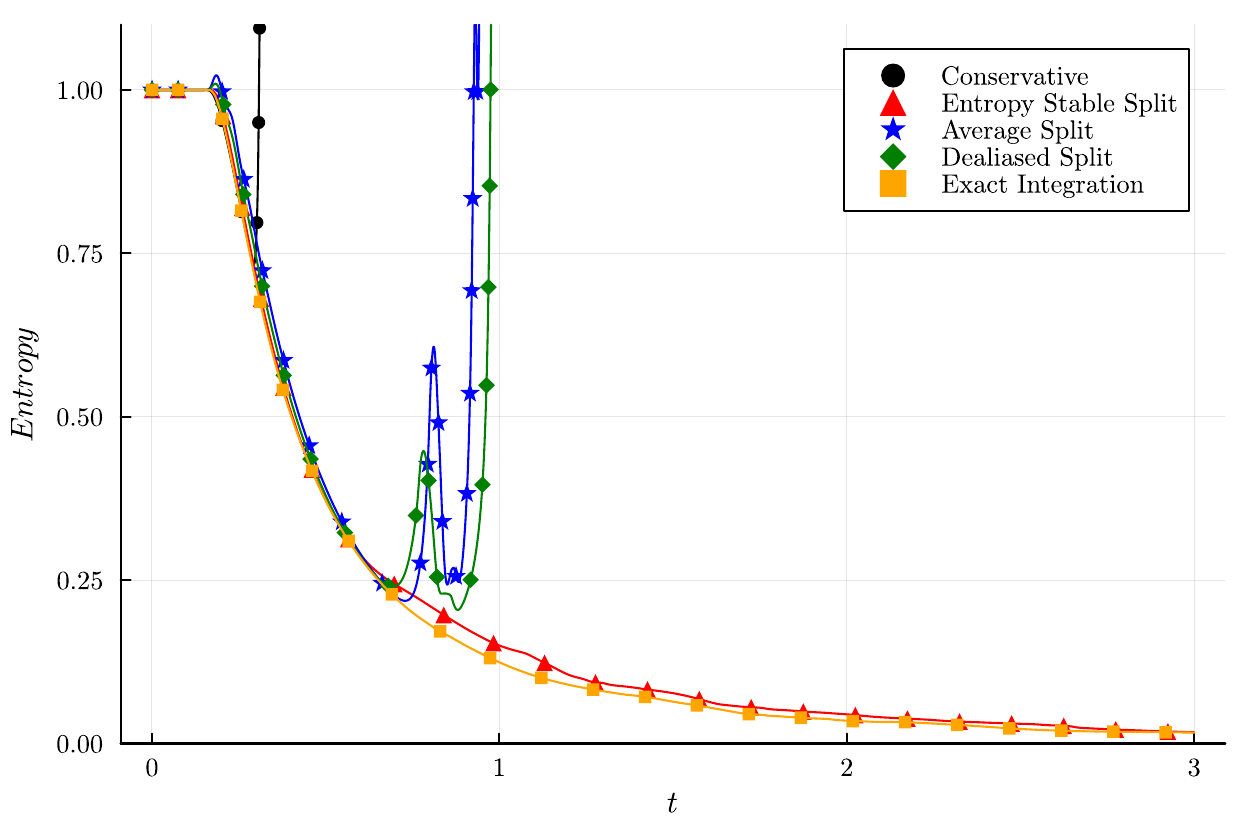}
        \caption{$p=5$.}
    \end{subfigure}
    \hfill
    \begin{subfigure}{0.49\textwidth}
        \centering
        \includegraphics[width=\linewidth]{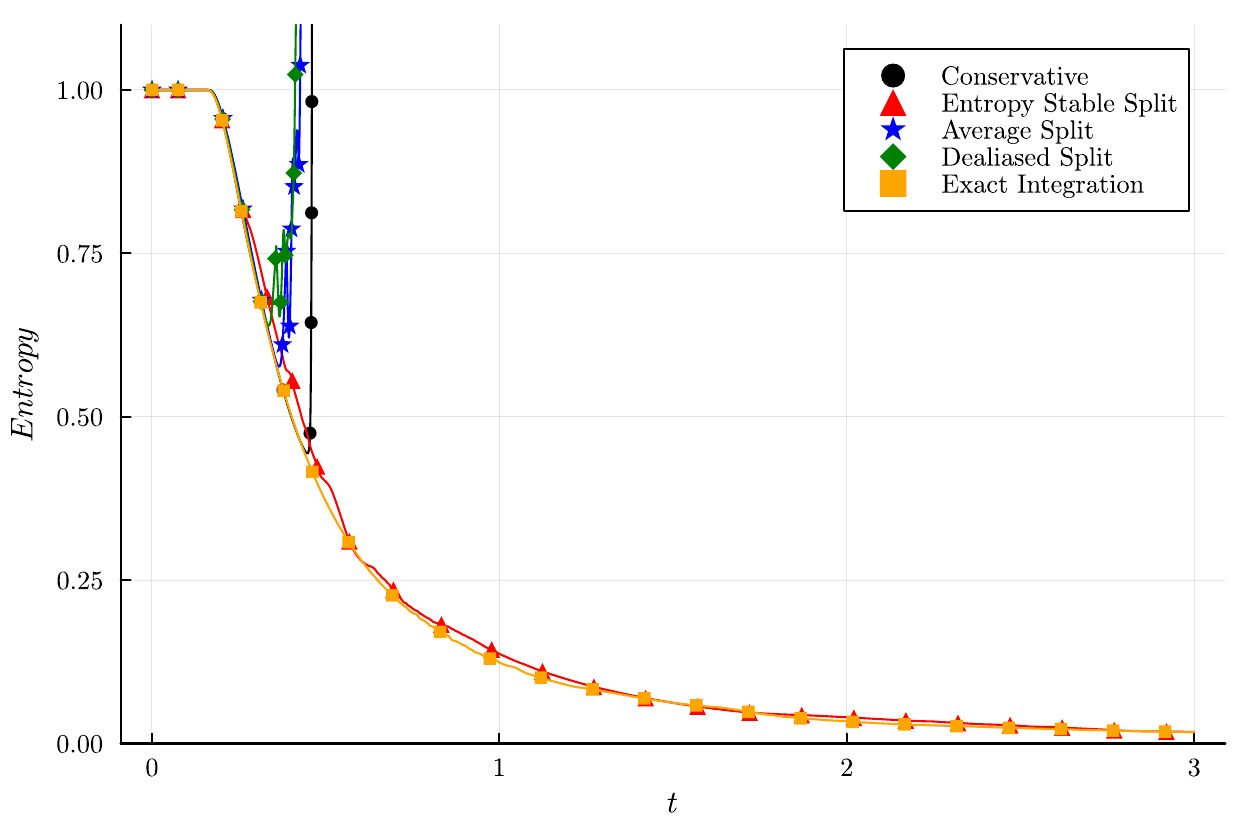}
        \caption{$p=8$.}
    \end{subfigure}
    \caption{Time evolution of the quadratic entropy for the different split forms using $(p+1)$ GLL quadrature nodes with an inexact mass matrix.}
    \label{fig:entropy_vs_time_GLL}
\end{figure}
\begin{figure}
    \centering
    \begin{subfigure}{0.49\textwidth}
        \centering
        \includegraphics[width=\linewidth]{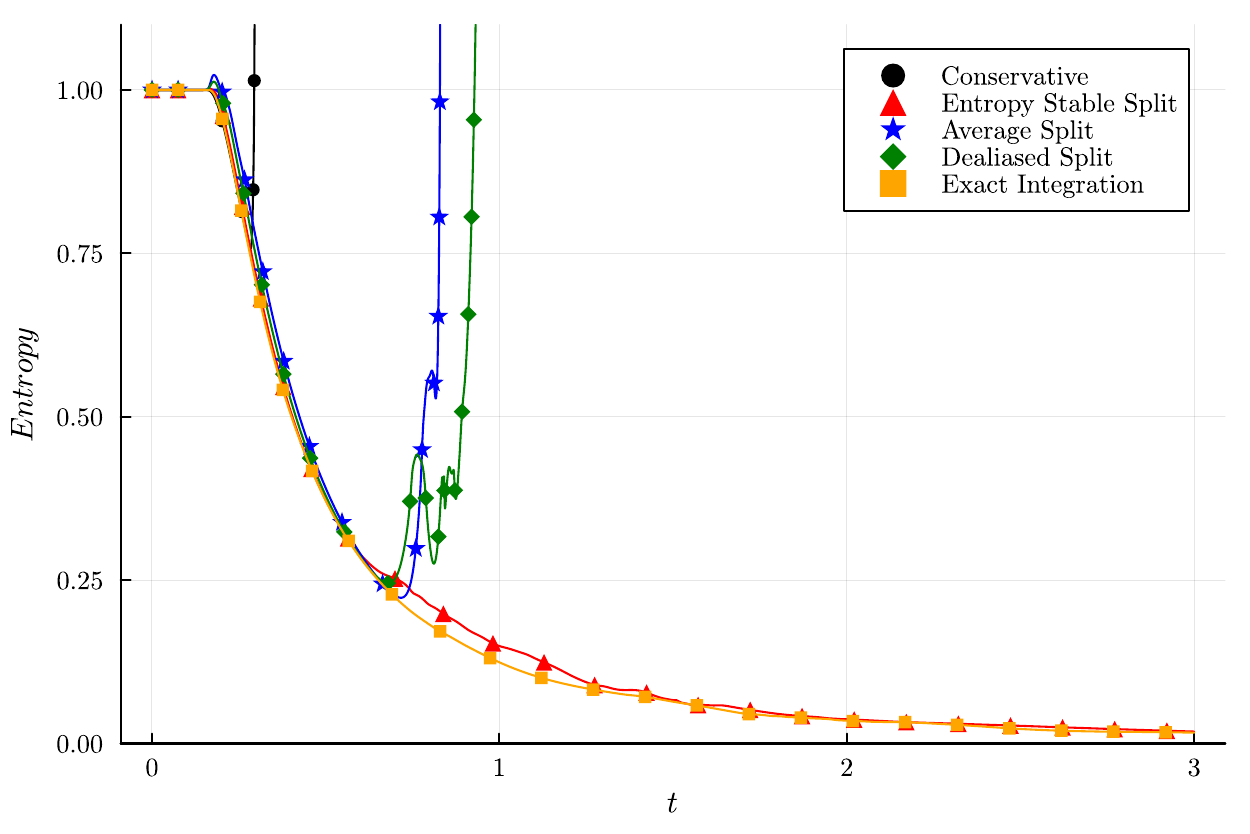}
        \caption{$p=5$.}
    \end{subfigure}
    \hfill
    \begin{subfigure}{0.49\textwidth}
        \centering
        \includegraphics[width=\linewidth]{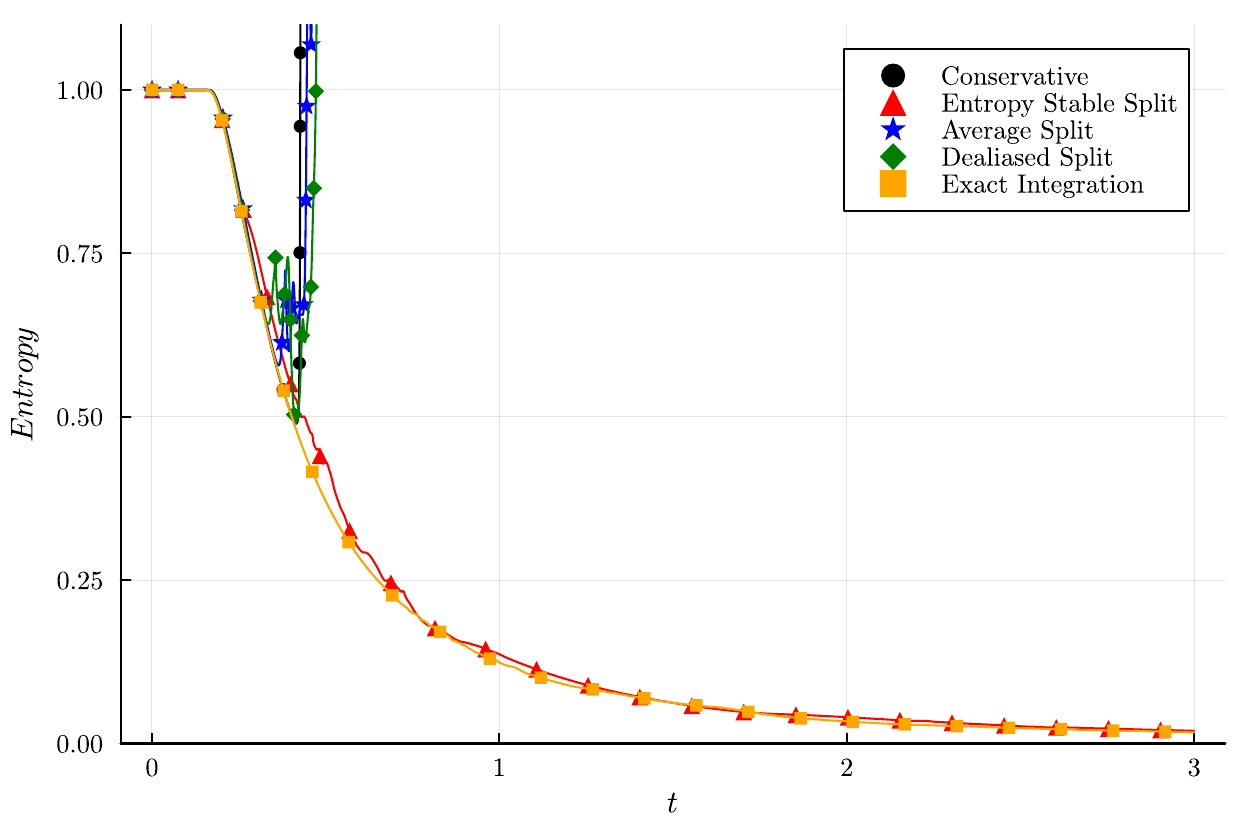}
        \caption{$p=8$.}
    \end{subfigure}
    \caption{Time evolution of the quadratic entropy for the different split forms using $(p+1)$ GLL quadrature nodes with an exact mass matrix.}
    \label{fig:entropy_vs_time_GLLEXACT}
\end{figure}
\begin{figure}
    \centering
    \begin{subfigure}{0.49\textwidth}
        \centering
        \includegraphics[width=\linewidth]{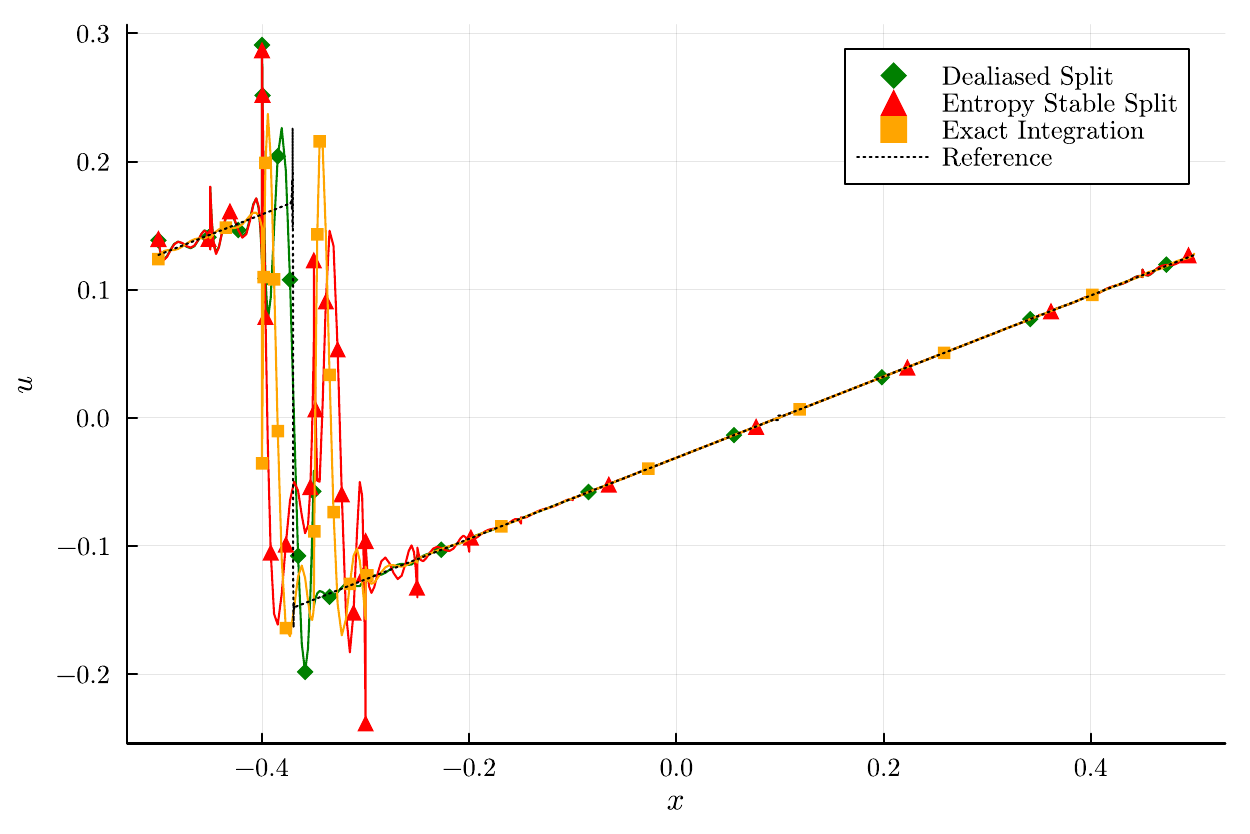}
        \caption{$p=5$.}
    \end{subfigure}
    \hfill
    \begin{subfigure}{0.49\textwidth}
        \centering
        \includegraphics[width=\linewidth]{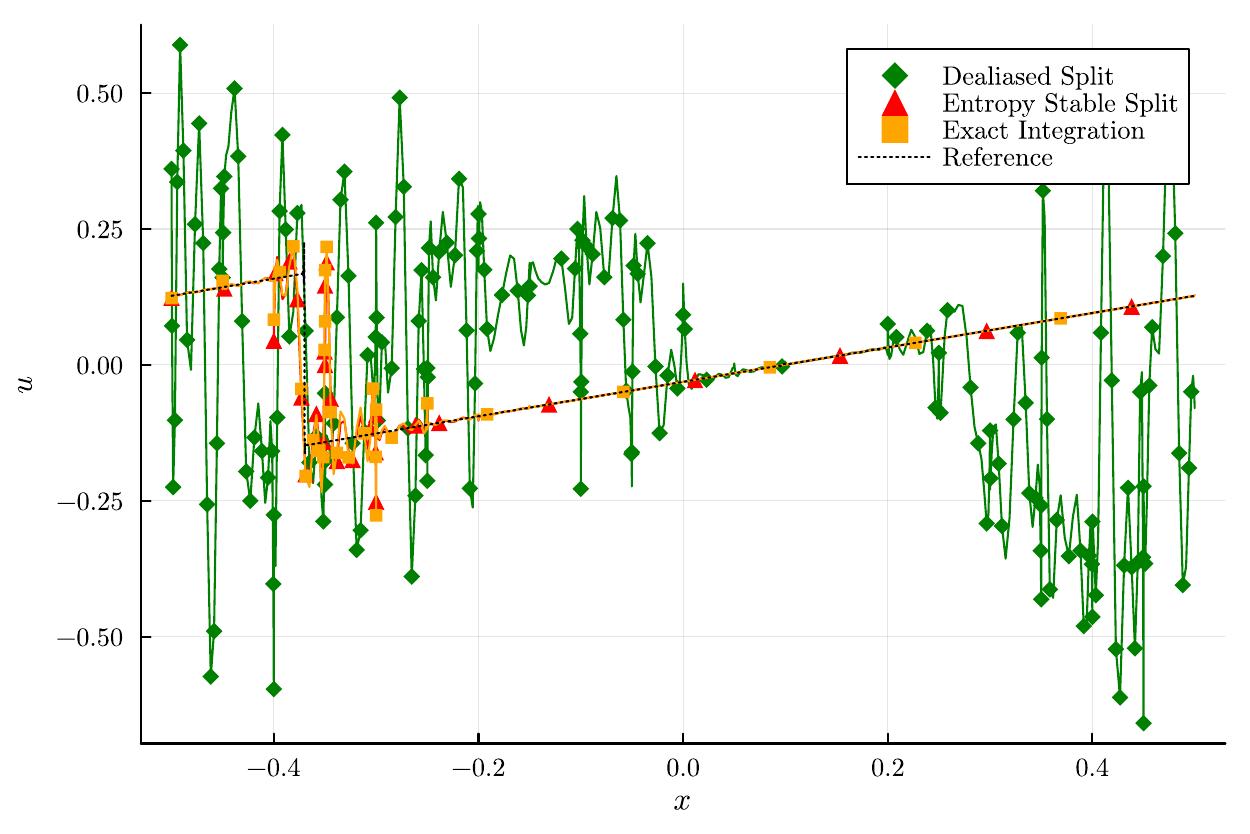}
        \caption{$p=8$.}
    \end{subfigure}
    \caption{Numerical solutions at $t=3.0$ for the different split forms using $(p+1)$ GL quadrature nodes.}
    \label{fig:sol_t3}
\end{figure}

\section{Conclusion}
{In this study, the dealiasing properties of split form DG discretizations of quadratic non-linearities were formally investigated. By introducing and studying the structural properties of the error tensor $\Delta_{ijk}(\alpha, \beta)$, we have shown that, akin to split form spectral schemes, split form DG discretizations achieve aliasing error reduction in the weakly underresolved range via integration error cancellation on the conservative and non-conservative forms, thereby generalizing the work of \cite{blaisdell1996effect} on spectral discretizations. While the dealiasing mechanism for split form DG discretizations is qualitatively similar to that observed by \cite{blaisdell1996effect} for spectral discretizations, it exhibits fundamental quantitative differences. In particular, as a byproduct of our analysis, quadrature- and order-dependent splitting coefficients that eliminate the dominant component of the aliasing error can be straightforwardly identified. These dealiased splitting coefficients are specific to DG discretizations. Through numerical experiments, it was confirmed that the dealiased splitting coefficients indeed minimize integration error in the weakly underresolved range. However, our theoretical and experimental results also show that entropy stability is required to prevent high-order components of the aliasing error from corrupting the solution when the integration time is sufficiently large. As demonstrated, this latter requirement is precisely fulfilled when the rank 3 error tensor is endowed with a skew-symmetric structure. One of the main novel contributions of this work is to translate the analysis of the aliasing errors associated with DG semi-discretization into the analysis of the structural properties of the error tensor $\Delta_{ijk}(\alpha, \beta)$. It is the opinion of the authors of this paper that this provides a conceptually elegant framework to understand and quantify the connection between aliasing errors, entropy stability and other components of the global truncation error.}

{Future work on this topic should primarily focus on extension of this study to DG discretizations of the Euler equations. While some of the arguments presented in this paper should be extendable to the Euler split formulations investigated in \cite{WINTERS20181}, it is unclear how this applies to general entropy stable discretizations written in the FD framework. More precisely, while entropy stable FD discretizations for the Euler equations share the same discrete operator structure as that considered in this work, two-point fluxes satisfying the Tadmor shuffle condition \cite{tadmor1987numerical} do not, in general, collapse to a single split formulation.}

{Additionally, the aliasing properties of entropy stable DG discretizations for the Burgers problem that do not rely on split formulations \cite{chan2018discretely, abgrall2022reinterpretation} should be assessed. Since such discretizations do not inherit their discrete operator structure from the FD framework, it would be relevant to understand the theoretical and numerical consequences of their associated aliasing errors.}

\section*{Acknowledgments}
The authors would like to thank the National Sciences and Engineering Research Council of Canada (NSERC) and McGill University for their financial support.

\section*{Competing Interests}
The authors have no conflicts of interest to declare that are relevant to the content of this paper.

\section*{AI Use}
The authors declare that no generative artificial intelligence tools were used in the preparation of this manuscript.

\appendix
\section{Continuous Split Form for Burgers Flux Differencing}
\label{app1}
The first part of the derivation is similar to that presented in Appendix by \cite{chan2018discretely}. We start with
\begin{equation}
    \mathbf{M}\frac{d\mathbf{u}}{dt} +
    \begin{bmatrix}
        \bm{\chi}_q \\
        \bm{\chi}_f
    \end{bmatrix}^T
    ((\mathbf{Q}_h - \mathbf{Q}_h^T) \circ \mathbf{F}) \mathbf{1}
    + \bm{\chi}_f^T \mathbf{B} \mathbf{f}^* = 0,
\end{equation}
with
\begin{equation}
    \mathbf{F} = \frac{1}{4}\alpha(\mathbf{1}\bar{\mathbf{u}}^T \circ \mathbf{1}\bar{\mathbf{u}}^T + \bar{\mathbf{u}}\mathbf{1}^T \circ \bar{\mathbf{u}}\mathbf{1}^T)
    + \frac{1}{2} \beta(\bar{\mathbf{u}}\mathbf{1}^T \circ \mathbf{1}\bar{\mathbf{u}}^T)
    \quad \text{where} \quad
    \bar{\mathbf{u}} :=
    \begin{bmatrix}
        \bm{\chi}_q\mathbf{u} \\ \bm{\chi}_f\mathbf{u}
    \end{bmatrix}.
\end{equation}
The skew-symmetric operator is given by
\begin{equation}
    \mathbf{Q}_h-\mathbf{Q}_h^T =
    \begin{bmatrix}
        \mathbf{P}^T\hat{\mathbf{Q}} \mathbf{P} - \mathbf{P}^T\hat{\mathbf{Q}}^T \mathbf{P} & \mathbf{P}^T \bm{\chi}_f^T \mathbf{B} \\
        -\mathbf{B} \bm{\chi}_f\mathbf{P} & 0
    \end{bmatrix}
    =
    \begin{bmatrix}
        2\mathbf{P}^T\hat{\mathbf{Q}} \mathbf{P} - \mathbf{P}^T\bm{\chi}_f^T\mathbf{B}\bm{\chi}_f\mathbf{P} & \mathbf{P}^T \bm{\chi}_f^T \mathbf{B} \\
        -\mathbf{B} \bm{\chi}_f\mathbf{P} & 0
    \end{bmatrix},
\end{equation}
where $\mathbf{P}:= \mathbf{M}^{-1}\bm{\chi}^T_q\mathbf{W}$ is the projection matrix from \cite{chan2018discretely}, $\mathbf{W}$ is a diagonal matrix storing the volume quadrature weights and $\hat{\mathbf{Q}}$ is the modal stiffness matrix. We first note that
\begin{align}
    \begin{bmatrix}
        \bm{\chi}_q \\
        \bm{\chi}_f
    \end{bmatrix}^T
    ((\mathbf{Q}_h - \mathbf{Q}_h^T) \circ \mathbf{F}) \mathbf{1}
    &=
    \begin{bmatrix}
        \bm{\chi}_q \\
        \bm{\chi}_f
    \end{bmatrix}^T
    \text{diag}((\mathbf{Q}_h - \mathbf{Q}_h^T)\mathbf{F}) \nonumber \\
    &=
    \begin{bmatrix}
        \bm{\chi}_q \\
        \bm{\chi}_f
    \end{bmatrix}^T
     \left(
     \frac{\alpha}{4}(\mathbf{Q}_h - \mathbf{Q}_h^T)(\bar{\mathbf{u}} \circ \bar{\mathbf{u}}) +
     \frac{\alpha}{4}((\mathbf{Q}_h - \mathbf{Q}_h^T)\mathbf{1}) \circ (\bar{\mathbf{u}} \circ \bar{\mathbf{u}}) +
     \frac{\beta}{2}\bar{\mathbf{u}} \circ (\mathbf{Q}_h - \mathbf{Q}_h^T)\bar{\mathbf{u}}
     \right)
\end{align}
For the first term, direct computation reveals
\begin{equation}
\frac{\alpha}{4}
\begin{bmatrix}
        \bm{\chi}_q \\
        \bm{\chi}_f
    \end{bmatrix}^T
    (\mathbf{Q}_h - \mathbf{Q}_h^T)(\bar{\mathbf{u}} \circ \bar{\mathbf{u}})
    =
    -\frac{\alpha}{2}\hat{\mathbf{Q}}\mathbf{P}(\bm{\chi}_q\mathbf{u})^2 + \frac{\alpha}{4}\bm{\chi}_f^T \mathbf{B} (\bm{\chi}_f\mathbf{u})^2.
\end{equation}
Similarly, for the second term, we have
\begin{equation}
\frac{\alpha}{4}
\begin{bmatrix}
        \bm{\chi}_q \\
        \bm{\chi}_f
    \end{bmatrix}^T
    ((\mathbf{Q}_h - \mathbf{Q}_h^T)\mathbf{1}) \circ (\bar{\mathbf{u}} \circ \bar{\mathbf{u}})
    = -\frac{\alpha}{4}\bm{\chi}_f^T \mathbf{B} (\bm{\chi}_f\mathbf{u})^2.
\end{equation}
Finally, for the third term, one finds
\begin{equation}
    {
    \frac{\beta}{2}
    \begin{bmatrix}
        \bm{\chi}_q \\
        \bm{\chi}_f
    \end{bmatrix}^T
    (\bar{\mathbf{u}} \circ (\mathbf{Q}_h - \mathbf{Q}_h^T)\bar{\mathbf{u}}))
    =
    \frac{\beta}{2}
    \begin{bmatrix}
        \bm{\chi}_q \\
        \bm{\chi}_f
    \end{bmatrix}^T
    \left(\bar{\mathbf{u}}
    \circ
    \begin{bmatrix}
        2\mathbf{P}^T \hat{\mathbf{Q}} \mathbf{u} \\
        -\mathbf{B}\bm{\chi}_f \mathbf{u}
    \end{bmatrix}\right)
    =
    \beta \bm{\chi}_q^T (\bm{\chi}_q\mathbf{u} \circ \mathbf{P}^T \hat{\mathbf{Q}} \mathbf{u}) - \frac{\beta}{2}\bm{\chi}_f ^T(\bm{\chi}_f \mathbf{u} \circ \mathbf{B}\bm{\chi}_f \mathbf{u}).
    }
\end{equation}
Putting all of this together, the scheme becomes
\begin{equation}
    {
    \mathbf{M}\frac{d\mathbf{u}}{dt}
    -\frac{\alpha}{2}\hat{\mathbf{Q}}\mathbf{P}(\bm{\chi}_q\mathbf{u})^2
    +\beta \bm{\chi}_q^T (\bm{\chi}_q\mathbf{u} \circ \mathbf{P}^T \hat{\mathbf{Q}} \mathbf{u}) - \frac{\beta}{2}\bm{\chi}_f ^T(\bm{\chi}_f \mathbf{u} \circ \mathbf{B}\bm{\chi}_f \mathbf{u})
    + \bm{\chi}_f^T \mathbf{B} \mathbf{f}^*
    = 0.
    }
\end{equation}
{In continuous form, this amounts to}
\begin{equation}
    \int_\Omega \frac{d}{dt}\Pi(\phi u) d\Omega
    - \alpha \frac{1}{2}\int_\Omega \frac{d \phi}{dx} \Pi(u^2) d\Omega
    + \beta \int_{\Omega} \frac{du}{dx}\Pi(\phi u) d\Omega
    - \beta \frac{1}{2}\int_{\Gamma} \phi u^2 \hat{n} d\Gamma
    + \int_{\Gamma} \phi f^* \hat{n} d\Gamma
    = 0,
\end{equation}
for every test function $\phi \in \mathcal{P}^p$.

\section{Main Aliased Band for Collocation Projection at the GLL and GL nodes}
\label{app2}

\subsection{GL Case}
By definition, the $(p+1)$ GL nodes are the roots of $L_{p+1}$. Hence,
\begin{equation}
    \Pi(L_{p+1}) = 0
    \implies
    \int_{\Omega} \Pi(L_{p+1}^2) d\Omega = 0.
\end{equation}
Moreover, from the Legendre product formula,
\begin{equation}
    L_{p+1}^2 =
    2
    \begin{pmatrix}
        p+1 & p+1 & 0 \\
        0 & 0 & 0
    \end{pmatrix}^2
    +
    (4p+5)
    \begin{pmatrix}
        p+1 & p+1 & 2p+2 \\
        0 & 0 & 0
    \end{pmatrix}^2L_{2p+2}
    +
    \sum_{r=1}^{2p+1}
    (2r+1)
    \begin{pmatrix}
        p+1 & p+1 & r \\
        0 & 0 & 0
    \end{pmatrix}^2L_{r}.
\end{equation}
Applying the projection and integrating both sides, we find
\begin{equation}
    0 =
    2
    \begin{pmatrix}
        0 & 0 & p+1 \\
        0 & 0 & 0
    \end{pmatrix}^2
    +
    (4p+5)
    \begin{pmatrix}
        p+1 & p+1 & 2p+2 \\
        0 & 0 & 0
    \end{pmatrix}^2
    \int_\Omega \Pi(L_{2p+2})d\Omega.
\end{equation}
Upon simplification of the Wigner $3j$ symbols via Lemma \ref{lem:Wigner1} and Lemma \ref{lem:Wigner2} (\ref{app3}), we find
\begin{equation}
    \int_{\Omega} \Pi(L_{2p+2})d\Omega
    =
    -\frac{2}{2p+3} \frac{
    \begin{pmatrix}
        4p+4 \\ 2p+2
    \end{pmatrix}
    }{
    \begin{pmatrix}
        2p+2 \\ p+1
    \end{pmatrix}^2
    }.
\end{equation}

\subsection{GLL Case}
By definition, the GLL nodes are the roots of $(x^2-1)L_p'(x)$. By the Bonnet recurrence formulas, we find
\begin{align}
    (x^2-1)L_p'(x)
    &= p(xL_p(x)-L_{p-1}(x)) \\
    &= p\left(\frac{(p+1)L_{p+1}(x)+pL_{p-1}(x)}{2p+1} - L_{p-1}(x)\right)
    = \frac{p(p+1)}{2p+1}(L_{p+1}(x)-L_{p-1}(x)).
\end{align}
Hence,
\begin{equation}
    \Pi(L_{p+1}) = \Pi(L_{p-1})
    \implies
    \Pi(L_{p+1}L_{p-1}) = \Pi(L_{p-1}^2).
\end{equation}
Integrating both sides,
\begin{equation}
    \int_{\Omega} \Pi(L_{p+1}L_{p-1}) d \Omega = \int_{\Omega} \Pi(L_{p-1}^2) d \Omega = \int_{\Omega} L_{p-1}^2 d \Omega = \frac{2}{2p-1}.
\end{equation}
From the Legendre product formula,
\begin{equation}
    L_{p-1} L_{p+1}
    =
    (4p+1)
    \begin{pmatrix}
        p-1 & p+1 & 2p \\
        0 & 0 & 0
    \end{pmatrix}^2
    L_{2p} +
    \sum_{r=2}^{2p-1} (2r+1)
    \begin{pmatrix}
        p-1 & p+1 & r \\
        0 & 0 & 0
    \end{pmatrix}^2
    L_r.
\end{equation}
Applying the projection and integrating both sides,
\begin{equation}
    \frac{2}{2p-1}
    =
    \int_{\Omega}\Pi(L_{p-1} L_{p+1})d\Omega
    =
    (4p+1)
    \begin{pmatrix}
        p-1 & p+1 & 2p \\
        0 & 0 & 0
    \end{pmatrix}^2
    \int_{\Omega}\Pi(L_{2p})d\Omega.
\end{equation}
Upon simplification of the Wigner $3j$ symbol via Lemma \ref{lem:Wigner1} (\ref{app3}), one finds
\begin{equation}
    \int_{\Omega} \Pi(L_{2p})d\Omega
    =
    \frac{2}{2p-1}
    \frac{
    \begin{pmatrix}
        4p \\ 2p
    \end{pmatrix}
    }
    {
    \begin{pmatrix}
        2p+2 \\ p+1
    \end{pmatrix}
    \begin{pmatrix}
        2p-2 \\ p-1
    \end{pmatrix}
    }.
\end{equation}

\section{Wigner 3j Symbol Simplification}
\label{app3}
{We compute the Wigner $3j$ symbol for two specific cases that are useful in the context of this work.
\begin{lemma}
Let $m,n \in \mathbb{N}$. Then,
    \begin{equation}
        \begin{pmatrix}
            m & n & m+n \\
            0 & 0 & 0
        \end{pmatrix}^2
        =
        \frac{1}{2m+2n+1}
        \frac{
        \begin{pmatrix}
            2m \\ m
        \end{pmatrix}
        \begin{pmatrix}
            2n \\ n
        \end{pmatrix}
        }{
        \begin{pmatrix}
            2m+2n \\ m+n
        \end{pmatrix}
        }
    \end{equation}
    \label{lem:Wigner1}
\end{lemma}
\begin{proof}
    Using Eq.(34.3.5) in \cite{NIST:DLMF} with $J=2m+2n$,
    \begin{align}
        \begin{pmatrix}
            m & n & m+n \\
            0 & 0 & 0
        \end{pmatrix}^2
        &=
        \frac{(J-2m)!(J-2n)!(J-2m-2n)!}{(J+1)!}
        \left(\frac{\left(\frac{1}{2}J\right)!}{\left(\frac{1}{2}J-m\right)!\left(\frac{1}{2}J-n\right)!\left(\frac{1}{2}J-m-n\right)!}\right)^2 \nonumber \\
        &=
        \frac{(2n)!(2m)!}{(2m+2n+1)(2m+2n)!}\left(\frac{(m+n)!}{n!m!}\right)^2
        = \frac{1}{2m+2n+1}
        \frac{
        \begin{pmatrix}
            2m \\ m
        \end{pmatrix}
        \begin{pmatrix}
            2n \\ n
        \end{pmatrix}
        }
        {
        \begin{pmatrix}
            2m + 2n \\ m+n
        \end{pmatrix}
        }.
    \end{align}
\end{proof}
\begin{lemma}
    Let $m \in \mathbb{N}$. Then,
    \begin{equation}
        \begin{pmatrix}
            m & m & 0 \\
            0 & 0 & 0
        \end{pmatrix}^2
        =
        \frac{1}{2m+1}.
    \end{equation}
    \label{lem:Wigner2}
\end{lemma}
\begin{proof}
    Using Eq.(34.3.5) in \cite{NIST:DLMF} with $J=2m$,
    \begin{equation}
        \begin{pmatrix}
            m & m & 0 \\
            0 & 0 & 0
        \end{pmatrix}^2
        =
        \frac{(J-2m)!(J-2m)!J!}{(J+1)!}
        \left(\frac{\left(\frac{1}{2}J\right)!}{\left(\frac{1}{2}J-m\right)!\left(\frac{1}{2}J-m\right)!\left(\frac{1}{2}J\right)!}\right)^2
        =
        \frac{1}{2m+1}.
    \end{equation}
\end{proof}
}

\section{Convergence Rate of the Integration Error}
\label{app4}
{We briefly sketch a derivation for the convergence rate of the integration error for a valid collocation projection operator associated with a quadrature strength $\tau$. We refer the reader to the work of \cite{https://doi.org/10.1002/num.22089} for a thorough derivation of error estimates for the DG formulation. Let $f \in \mathcal{P}^{M}$ for $M \geq \tau -p + 2$, $\phi \in \mathcal{P}^p$ and $\Omega_x$ and $\Omega$ denote the physical element and the reference element respectively. Moreover, we let $x$ denote the spatial coordinate on the physical element and $\xi$ represent the latter on the reference element, and assume that $\Omega$ and $\Omega_x$ are related via a linear mapping. The conservative DG discretization constructs a discrete approximation $g \in \mathcal{P}^p$ via
\begin{equation}
   \int_{\Omega_x}\phi g d\Omega_x =\int_{\Omega_x} \frac{d\phi}{dx} \Pi(f) d\Omega_x = \int_{\Omega_x} \Pi\left(\frac{d\phi}{dx}f\right) d\Omega_x.
   \label{eq:DG_conser_rate}
\end{equation}
On the reference element, this can be written as
\begin{equation}
   \int_{\Omega}\phi g \Delta x d\Omega = \int_{\Omega} \Pi\left(\frac{d\phi}{dx}f\right) d\Omega,
\end{equation}
where $\Delta x := |\Omega_x|$. Moreover, on the reference element, $f$ can be written as
\begin{equation}
    f(\xi) = \sum_{i=0}^{M} (\Delta x)^i f_i (\xi-c)^i = \sum_{i=0}^{\tau-p+1} (\Delta x)^i f_i (\xi-c)^i + \sum_{i=\tau-p+2}^{M} (\Delta x)^i f_i (\xi-c)^i,
\end{equation}
where $c \in \mathbb{R}$ is a constant resulting from the linear mapping of the reference element to the physical element. Substituting this in Eq.(\ref{eq:DG_conser_rate}), we find
\begin{equation}
    \int_{\Omega}\phi g d\Omega = \int_{\Omega} \Pi\left(\frac{d\phi}{dx}\sum_{i=0}^{\tau-p+1} (\Delta x)^{i-1} f_i (\xi-c)^i\right) d\Omega +
    \int_{\Omega} \Pi\left(\frac{d\phi}{dx}\sum_{i=\tau-p+2}^{M} (\Delta x)^{i-1} f_i (\xi-c)^i\right) d\Omega .
\end{equation}
The leftmost term is integrated exactly, while the rightmost term will lead to aliasing errors of order $O(\Delta x^{\tau-p+1})$. The same argument can be applied to the non-conservative discretization to find that the associated aliasing errors scale with $O(\Delta x^{\tau-p+1})$. Hence, provided that the leading-order error term from the conservative and non-conservative discretizations do not cancel, the aliasing errors on any split form DG discretization should also scale with $O(\Delta x^{\tau-p+1})$. From Corollary \ref{corr:optimal_splits}, it is known that such a cancellation will occur if and only if the suitable dealiased splitting coefficients are utilized. Moreover, from Theorem \ref{theo:Delta}, it is known that the dealiased split coefficients do not lead to cancellation of the aliasing error for higher-order components of the flux. Hence, the aliasing error will scale with $O(\Delta x^{\tau-p+2})$ in this case.}

 \bibliographystyle{elsarticle-num} 
 \bibliography{bib.bib}




\end{document}